\documentclass{article}

\title{$P$-adic $L$-functions: $t$-modules and Dirichlet-Goss $L$-series}

\date{}
\author{Daniel Krell Calvo }

\usepackage[english]{babel}
\usepackage[T1]{fontenc}
\usepackage[utf8]{inputenc}
\usepackage{amsfonts}
 \usepackage{mathrsfs} 
\usepackage{amsmath}
\usepackage{amsthm}
\usepackage{amssymb}
\usepackage{url}
\usepackage{hyperref}
\usepackage{enumerate}
\usepackage{dsfont}
\usepackage{moreenum}
\usepackage{mathtools}
\usepackage{verbatim}

  \newcommand{\N}{\mathbb{N}}     % Naturales
  \newcommand{\C}{\mathbb{C}}     % Complejos
  \newcommand{\F}{\mathbb{F}}	  % Cuerpo finito
  \newcommand{\Fq}{\mathbb{F}_q}	  % Cuerpo finito
  \newcommand{\D}{\mathcal{D}}	 
  
   \newcommand{\OF}{\mathcal{O}_F}  % Ring of integers
   \newcommand{\OFs}{\mathcal{O}_{F,s}}

  \newcommand{\sep}{\text{sep}}
    \newcommand{\T}{\mathbb{T}}     % Tate series
    \newcommand{\Cn}{C^{\otimes n}}

    \newcommand{\Ksi}{K_{s,\infty}}
    \newcommand{\Fqchi}{\F_q(\chi)}
    \newcommand{\Fqs}{\F_q(\underline{t})}

    \newcommand{\Et}{\widetilde{E}}

   \newcommand{\expEt}{\exp_{\Et}}
   \newcommand{\LieE}{\Lie_E}
  \newcommand{\LieEt}{\Lie_{\Et}}

  \newcommand{\Fsi}{F_{s,\infty}}

\newtheorem{theorem}{Theorem}
\newtheorem{theoremA}{Theorem}

\newtheorem{lemma}[theorem]{Lemma}
\newtheorem{prop}[theorem]{Proposition}
\newtheorem*{prop*}{Proposition}
\newtheorem{cor}[theorem]{Corollary}

\newtheorem*{question}{Question}
\newtheorem*{claim}{Claim}

\theoremstyle{remark}
\newtheorem{remark}{Remark}

\theoremstyle{definition}
\newtheorem{defin}[theorem]{Definition}

\DeclareMathOperator{\sgn}{sgn}

\DeclareMathOperator{\ord}{ord}

\DeclareMathOperator{\Mat}{Mat}

\DeclareMathOperator{\Lie}{Lie}

\DeclareMathOperator{\I}{I}
\DeclareMathOperator{\lcm}{lcm}

\DeclareMathOperator{\Log}{Log}

\DeclareMathOperator{\lc}{lc}
\DeclareMathOperator{\cond}{cond}

\hypersetup{
  colorlinks   = true, %Colours links instead of ugly boxes
  urlcolor     = black, %Colour for external hyperlinks
  linkcolor    = blue, %Colour of internal links
  citecolor   = green %Colour of citations
}

\makeindex

\begin{document}

\maketitle

\begin{abstract}
    We prove that the order of vanishing of a natural twist of the $P$-adic Carlitz zeta values at the positive "even" integers is always $1$. We also obtain the same result for $P$-adic Pellarin $L$-series, and for almost all $P$-adic Dirichlet-Goss $L$-series. To do so, we demonstrate a $P$-adic version of a reduced variant of the class formula for Anderson $t$-modules proven by Anglès, Ngo Dac and Tavares-Ribeiro in 2020.
\end{abstract}

\tableofcontents

\section{Introduction}
Let $\F_q$ be a finite field with $q$ elements, where $q$ is a power of a prime number. Let $A=\F_q[\theta]$ be a polynomial ring, and let $A_+$ be the subset of monic polynomials. The Carlitz zeta values
$$\zeta(n):=\sum_{a\in A_+}\frac{1}{a^n},\qquad n\in \N,$$
are the function field equivalent to the values at the positive integers of the Riemann zeta function. In \cite{carlitz_certain_1935}, Carlitz discovered that there exists a link between the special value at $1$ and a certain $A$-module, which is now known as the Carlitz module. More specifically, we can attach a logarithm $\log_C$ to the Carlitz module, and we have
\begin{equation}\label{eq : Carlitz identity}
    \zeta(1)=\log_{C}(1).
\end{equation}

This phenomenon happens in much greater generality. It is a particular case of what we now call the \textit{class formula}.

Let $K=\F_q(\theta)$, let $F/K$ be a finite extension, and let $\OF$ be the integral closure of $A$ in $F$. Let $K_\infty$ be the completion of $K$ with respect to the place of infinity $v_\infty=-\deg_{\theta}$, and let $\tau:x\mapsto x^q$ denote the Frobenius morphism. An Anderson $A$-module (or Anderson module in short) $E/\OF$ of dimension $n$ is a morphism of $\F_q$-algebras $E:A \rightarrow\Mat_{n\times n}(\OF)\{\tau\}$ such that, if for any $a\in A$ we write 
$$E_a=\partial_E(a)+\sum_{i\geq 1}E_{a,i}\tau^i,$$ 
then we require $(\partial_E(a)-a \I_n)^n=0$.

Given an $\OF$-algebra $B$, we can define two new $A$-module structures on $B^n$: one denoted by $E(B)$, where $A$ acts on $B^d$ via $E$, and one denoted by $\LieE(B)$, where $A$ acts on $B^n$ via $\partial_E$.

 To every Anderson module $E$, we can associate two special series living in $\Mat_{n\times n}(K)\{\{\tau\}\}$: an exponential $\exp_E$ and a logarithm $\log_E$. These series induce formal morphisms between the structures $E(K_\infty)$ and $\LieE(K_\infty)$ (see Section \ref{sec : Background and notation}).
 
We define the $L$-series (or $L$-value) of an Anderson $A$-module $E/\OF$ as the convergent infinite product
$$L(E/\OF):= \prod_{P}\frac{\left[\Lie_E\left(\frac{\OF}{P\OF}]\right)\right]_A}{\left[E\left(\frac{\OF}{P\OF}\right)\right]_A}\in K_\infty^\times,$$
where $P$ runs trough all monic irreducible elements of $A$, and $[\cdot]_A$ denotes the monic generator of the Fitting ideal. We note that in this case, the infinite product never vanishes. The \textit{class formula}  
\begin{equation*}
    L(E/\OF)=[\Lie_E(\OF) : U(E/\OF)]_A \cdot [H(E/\OF)]_A
\end{equation*}
provides a link between the $L$-series of $E$ and some other quantities associated to $E$. Of the two terms in the right hand-side, we will be mainly interested in the regulator
$$[\Lie_E(\OF) : U(E/\OF)]_A\in K_\infty,$$
which is a ratio of covolumes of two lattices of rank $n$. The other term $[H(E/\OF)]_A$, which is the Fitting ideal of the so-called class module, is an element that lives in $A$. This formula was first proven by Taelman in \cite{taelman_special_2012} for Drinfeld modules (\textit{i.e.} Anderson modules of dimension $1$), and later on by Fang \cite{fang_special_2015} and Demeslay \cite{demeslay_class_2022} for Anderson modules. A $P$-adic version of the class formula for Anderson modules was proven by Lucas in \cite{lucas_P-adic_2026}.

Before moving on, we will quickly highlight an analogy with the classical characteristic $0$ theory. The class formula for Anderson modules can be thought of as an analogue of the class number formula for number fields. We recall that if $k$ is a number field and $\zeta_k(n)$ denotes its corresponding Dedekind zeta function, then the class number formula
$$\lim_{s \to 1} (s-1)\,\zeta_k(s)
\;=\;
\frac{2^{r_1}\,(2\pi)^{r_2}\,\,\operatorname{Reg}_k}{w_k\,\sqrt{|D_k|}}h_k$$
gives us a link between the zeta function and some other quantities associated to $k$, including the regulator $\operatorname{Reg}_k$ and the class number $h_k$. We refer to \cite[Corollary 5.11]{neukirch_algebraic_1999} for more details.

Let us return to the function field case. The most prominent example of Drinfeld modules is the Carlitz module, which we denote by $C$, and was already introduced at the beginning of this section. It can be shown that $L(C/A)=\zeta(1)$, and that Equation \eqref{eq : Carlitz identity} is exactly the class formula for $C$. More generally, the $n$-th tensor power of the Carlitz module, which we denote by $\Cn$, is an Anderson module of dimension $n$, with the property that $L(\Cn/A)=\zeta(n)$. The class formula then gives us a link between $\zeta(n)$ and $\Cn$.

 For a family of Anderson modules including $\Cn$, Anglès, Ngo Dac and Tavares-Ribeiro showed in \cite{angles_special_2020} a reduced version of the class formula. They proved that there exists a vector space $W$ such that
\begin{equation} \label{eq : class formula modulo K}
    L(E/\OF) = [\LieE(\OF)\cap W : U(E/\OF)\cap W]_{A} \pmod{K^*},
\end{equation}
where $\LieE(\OF)\cap W$ and $U(E/\OF)\cap W$ are now lattices of smaller rank. The main goal of this article is to prove a $P$-adic version of this result, which will allow us to obtain some properties about certain $P$-adic $L$-series.

Let $P$ be an irreducible element of $A_+$, and let $K_P$ be the $P$-adic completion of $K$. We define the $P$-adic $L$-series of $E/\OF$ as
$$L_P(E/\OF):=\prod_{Q\neq P}\frac{[\LieE(\OF/Q\OF)]_A}{[E(\OF/Q\OF)]_A} \in K_P,$$
where $Q$ runs over all monic irreducible elements of $A$ different from $P$. Unlike in the $\infty$-adic case, the series $L_P(E/\OF)$ can vanish.

In order to prove a $P$-adic version of \eqref{eq : class formula modulo K}, we will first derive a $z$-twisted version, and then specialize $P$-adically at $z=1$. Let us clarify what we mean by this. Let $z$ be a new variable that commutes with everything. Let $\widetilde{A}=\F_q(z)A$ and $\widetilde{\OF}=\F_q(z)\OF$. Given an Anderson module $E/\OF$, we define the corresponding twisted module as the $\F_q(z)$-algebra morphism $\widetilde{E}:\widetilde{A} \rightarrow \Mat_{n\times n}(\widetilde{\OF})\{\tau\}$ given by
$$\widetilde{E}_a = \partial_E(a)+\sum_{i\geq 1}E_{a,i}z^i\tau^i.$$

The twisted $L$-values $L(\widetilde{E}/ \widetilde{\OF})$ and $L_P(\widetilde{E}/ \widetilde{\OF})$ are defined analogously to $L(E/\OF)$ and $L_P(E/\OF)$, and we have
$$L(\widetilde{E}/\widetilde{\OF})|_{z=1}=L(E/\OF),\qquad L_P(\widetilde{E}/\widetilde{\OF})|_{z=1}=L_P(E/\OF).$$
 The key remark is that both $L(\widetilde{E}/\widetilde{\OF})$ and $L_P(\widetilde{E}/\widetilde{\OF})$ live in $K[[z]]$. Thus, we can compare both series over $K[[z]]$, where $L(\widetilde{E}/\widetilde{\OF})$ equals $L_P(\widetilde{E}/\widetilde{\OF})$ times the local factor at $P$.

As discussed, we will obtain a twisted version of \eqref{eq : class formula modulo K}, linking $L(\widetilde{E}/ \widetilde{\OF})$ to the corresponding twisted objects. The proof of this part is formally the same as in the case without the twist. All elements of the twisted formula will live in $K((z))$ and have no $P$-adic poles at $z=1$. By evaluating at $z=1$ and looking at the $P$-adic convergence we will obtain the $P$-adic formula, which we now state. We note that all objects with tildes on top are suitable $z$-twists of the original objects. 

\begin{theoremA}[Theorem \ref{th : P-adic formula W}] \label{th A}
    Let $E/\OF$ be an $A$-finite Anderson module such that $\sigma  N_E(L)\subseteq (t-\theta)N_E(L)$. Over $\widetilde{K_P}$, we have the equality
    $$L_P(\widetilde{E}/\widetilde{\OF})= [\LieEt(\widetilde{\OF})\cap \widetilde{W} : U(\widetilde{E}/\widetilde{\OF})\cap \widetilde{W}]_{\widetilde{A},P} \pmod{\widetilde{K}^*}.$$
    Over $K_P$, we have
    $$L_P(E/\OF)= [\LieE(\OF)\cap W : U(E/\OF)\cap W]_{A,P}\pmod{K^*}.$$
\end{theoremA}
In Theorem \ref{th A},
$$[\LieE(\OF)\cap W : U(E/\OF)\cap W]_{A,P}\in K_P$$
is a $P$-adic version of the ratio of covolumes. It is particularly easy to compute when the lattices $\LieE(\OF)\cap W$ and $U(E/\OF)\cap W$ are of rank $1$. In such cases, we will be able to obtain some new results about the \textit{order of vanishing} of the $P$-adic $L$-series of $E$. We recall that $L_P(\widetilde{E}/\widetilde{\OF})|_{z=1}=L_P(E/\OF)$. Thus, if $L_P(E/\OF)$ vanishes, a natural follow-up question would be to ask what the order of vanishing of $L_P(\widetilde{E}/\widetilde{\OF})$ at $z=1$ is. 

\begin{theoremA}[Corollary \ref{cor : vanishing P-adic L} and Theorem \ref{th : vanishing does not depend P}]\label{th B}
     Let $E/A$ be an $A$-finite Anderson module of rank $1$ such that $\sigma  N_E(L)\subseteq (t-\theta)N_E(L)$. Then
     \begin{enumerate}
         \item $L_P(E/A)=0$ if and only if
    $\exp_E:\LieE(K_\infty)\rightarrow E(K_\infty)$
    is not injective.
         \item The order of vanishing of $L_P(\widetilde{E}/\widetilde{A})$ at $z=1$ does not depend on $P$.
     \end{enumerate}
\end{theoremA}

In \cite[Theorem 5.2]{lucas_P-adic_2026}, Lucas shows that for any Anderson module $E$, if the exponential $\exp_E$ is not injective, then $L_P(E/A)=0$. He then conjectures that the converse implication is also true. The first part of Theorem \ref{th B} gives a positive answer to this conjecture for a family of Anderson modules.

In the same paper, Lucas also rewrites in the language of Anderson modules a conjecture of Caruso and Gazda (\cite{caruso_computation_2025}), which was originally stated in the language of Anderson $t$-motives. The conjecture claims that for any Anderson module $E$, the order of vanishing at $z=1$ of $L_P(\widetilde{E}/\widetilde{\OF})$ does not depend on $P$. Thus, the second part of Theorem \ref{th B} gives a positive answer to this conjecture for a family of Anderson modules.

Let us go back to our original example, the Carlitz zeta values. The $P$-adic Carlitz zeta values are defined as
$$\zeta_P(n):= \sum_{d\geq 0}\sum_{\substack{a\in A_{+,d}\\ P\nmid a}} \frac{1}{a^n},\qquad n\in \N,$$
where $A_{+,d}$ denotes the set of monic polynomials of degree $d$ in $A$. It is a well-known fact that
$$\zeta_P(n)=0 \iff q-1 \mid n.$$
Yu actually proved in \cite{yu_transcendence_1991} that $\zeta_P(n)$ is transcendental whenever it is nonzero. If $q-1\mid n$, a natural follow-up question would be to compute the order of vanishing of
$$\zeta_P(n,z):= \sum_{d\geq 0}\sum_{\substack{a\in A_{+,d}\\ P\nmid a}} \frac{1}{a^n}z^d,\qquad n\in \N,$$
at $z=1$. 

As in the $\infty$-adic case, we have $L_P(\Cn/A)=\zeta_P(n)$ and $L_P(\widetilde{\Cn}/\widetilde{A})=\zeta_P(n,z)$.  Theorem \ref{th B} then tells us that the order of vanishing at $z=1$ of $\zeta_P(n,z)$ does not depend on $P$. Moreover, since Diaz-Vargas and Polanco-Chi proved in \cite{diaz-vargas_riemann_2016} that $z\mapsto \zeta_\theta(n,z)$ only has simple zeros, we can give a definite answer to the question of the order of vanishing of $\zeta_P(n)$.
\begin{theoremA}[Theorem \ref{th : order vanishing 1 riemann hyp}]
   If $q-1 \mid  n$, then 
    $$\ord_{z=1}\zeta_P(n,z)=1.$$
\end{theoremA}

Additionally, by doing some concrete computations on the twisted equation of Theorem \ref{th A} for $E=\Cn$, and using a special point recently found by Pellarin in \cite{pellarin_carlitz_2025}, we are able to derive a formula for the evaluation at $1$ of the derivative of $\zeta_P(n,z)$. 
Let
$$\zeta_P^{(1)}(n):=\left(\frac{d}{dz}\zeta_A(n,z)_P\right)|_{z=1}.$$ 
Set 
$$\mathcal{X}_l:= \left(-\binom{l}{n-i}(\theta-\theta^q)^{l-(n-i)}\right)_{1\leq i \leq n} \in \Mat_{n\times 1}(K),$$
$$e_n:=(0,0,\ldots,1)^T \in \Mat_{n\times 1}(K),$$
and let $\iota$ be the projection onto the last component. 
\begin{theoremA}[Theorem \ref{th : formula derivative Cn}]
    Suppose $n=l(q-1)$ with $l\in \N$. If $q\neq 2$, then
    $$\zeta_P^{(1)}(n) = \frac{(P^n-1)BC_n}{P^n\Gamma_n}\iota(\Log_{\Cn,P}(\mathcal{X}_l)).$$
    If $q=2$, then
    $$\zeta_P^{(1)}(n) = \frac{(P^n+1)BC_n}{P^n\Gamma_n}\iota(\Log_{\Cn,P}(\mathcal{X}_l+e_n)).$$
\end{theoremA}

Next, we use Theorems \ref{th A} and \ref{th B} to obtain similar results for a larger family of $P$-adic $L$-series. In Section \ref{sec : Pellarin L-series}, we devote ourselves to the study of the $P$-adic version of Pellarin $L$-series, which where introduced by Pellarin in \cite{pellerin_values_2012}. Let $t_1,\ldots,t_s$ be new variables, and let $\chi_{t_i}:A\rightarrow \F_q[t_1,\ldots,t_s]$ be the $\F_q$-algebra morphism defined by $\chi_{t_i}(\theta)=t_i$. Define the corresponding $P$-adic Pellarin $L$-series as
$$L_P(\chi_{t_1},\ldots,\chi_{t_s},n,z):=\sum_{d\geq 0}\sum_{\substack{a\in A_{+,d} \\ P\nmid a}}\frac{\chi_{t_1}(a)\ldots\chi_{t_s}(a)}{a^n} z^d,$$
$$L_P(\chi_{t_1},\ldots,\chi_{t_s},n) := L_P(\chi_{t_1},\ldots,\chi_{t_s},n,1).$$
 These $L$-series turn out to be the $L$-series corresponding to a certain Anderson module $E_\alpha$, which can be thought of as a multivariable twist of $\Cn$. Since $E_\alpha$ also satisfies the hypothesis of Theorem \ref{th B}, we have that $L_P(\chi_{t_1},\ldots,\chi_{t_s},n)$ vanishes if and only if the exponential of $E_\alpha$ is not injective (which happens if and only if $n\equiv s \pmod{q-1}$), and the order of vanishing of $L_P(\chi_{t_1},\ldots,\chi_{t_s},n,z)$ at $z=1$ does not depend on $P$. By some concrete computations in the case $P=\theta$, we are able to obtain the following result.  

\begin{theoremA} [Corollary \ref{cor : vanishing LP t1 ts} and Theorem \ref{th : order vanishing t1 ts}]
    We have that
    $$L_P(\chi_{t_1},\ldots,\chi_{t_s},n)=0\quad \text{ if and only if } \quad n\equiv s \pmod{q-1}.$$
    Moreover, if $n\equiv s \pmod{q-1}$, then 
        $$\ord_{z=1}L_P(\chi_{t_1},\ldots,\chi_{t_s},n,z) = 1.$$
\end{theoremA}

Finally, in Section \ref{sec : Dirichlet characters} we study the $P$-adic version of Dirichlet-Goss $L$-series (see \cite[Chapter 8]{goss_basic_1996}). A Dirichlet character $\chi$ of type $s$ is a product
$$\chi=\chi_{\eta_1}\ldots\chi_{\eta_s},$$ 
where for all $i$, $\eta_i\in \overline{\F_q}$, and $\chi_{\eta_i}:A \rightarrow \overline{\F_q}$ is the $\F_q$-algebra morphism given by $\chi_{\eta_i}(\theta)=\eta_i$. Define the corresponding $P$-adic Dirichlet-Goss $L$-series as
$$L_P(n,\chi,z):=\sum_{d\geq 0}\sum_{\substack{a\in A_{+,d}\\P\nmid a}}\frac{\chi(a)}{a^n}z^d,$$
$$L_P(n,\chi):=L_P(n,\chi,1).$$
Again, by Theorem \ref{th B} we have that the order of vanishing does not depend on $P$. In this case, however, the concrete computations in the case $P=\theta$ only allow as to give a satisfactory answer for \textit{almost all} Dirichlet characters.

\begin{theoremA}[Corollary \ref{cor : vanishing L Dirichlet} and Theorem \ref{th : ord vanishing dirichlet characters}]
Let $\chi$ be a Dirichlet character of type $s$. Then 
    $$L_P(n,\chi)=0 \quad \text{ if and only if } \quad n\equiv s \pmod{q-1}.$$
    Moreover, if $s\in \N$ such that $n\equiv s \pmod{q-1}$, then for \textit{almost all} Dirichlet characters $\chi$ of type $s$, we have
    $$\ord_{z=1}L_P(n,\chi,z) = 1.$$
    If $s<q-1$, then this is true for all characters of type $s$.
\end{theoremA}

Using our methods, we have not been able to compute $\ord_{z=1}L_P(n,\chi,z)$ for all Dirichlet characters. It is then natural to ask the following question.

\begin{question}
    Let $n\equiv s \pmod{q-1}$. Does
    $$\ord_{z=1}L_P(n,\chi,z) = 1$$
    for \textit{all} Dirichlet characters of type $s$?
\end{question}

We note that a positive answer to this question was given by Anglès and Tavares-Ribeiro in \cite{angles_arithmetic_2021} for the case $n=1$.

\paragraph{Acknowledgments.} This paper was written as a part of the author's PhD thesis at the University of Caen Normandy, under the supervision of Bruno Anglès and Floric Tavares-Ribeiro. The author would like to thank them for their guidance, support, and valuable discussions during the preparation of this work.

%%%%%%%%%%%%%%%%%%%%%%%%%%%%%%%%%%%%%%%%%%%%%%%%%%%%%%%%%%%%%%%%%%%%%%%%%%%%%%%%%%%%%%%%%%%%%%%%%%%%%%%%%%%%%%%%%%
%%%%%%%%%%%%%%%%%%%%%%%%%%%%%%%%%%%%%%%%%%%%%%%%%%%%%%%%%%%%%%%%%%%%%%%%%%%%%%%%%%%%%%%%%%%%%%%%%%%%%%%%%%%%%%%%%%
%%%%%%%%%%%%%%%%%%%%%%%%%%%%%%%%%%%%%%%%%%%%%%%%%%%%%%%%%%%%%%%%%%%%%%%%%%%%%%%%%%%%%%%%%%%%%%%%%%%%%%%%%%%%%%%%%%
%%% BACKGROUND AND NOTATION 

\section{Background and notation} \label{sec : Background and notation}
We fix the following notation:
\begin{itemize}
    \item $\F_q$ is a finite field with $q$ elements, where $q$ is a power of a prime $p$;
    \item $A:=\F_q[\theta]$;
    \item $A_d$ is the set of polynomials of $A$ of degree $d$;
    \item $A_+$ is the set of monic polynomials of $A$;
    \item $A_{+,d}$ is the set of monic polynomials of $A$ of degree $d$;
    \item $K:=\F_q(\theta)$;
    \item $\infty$ is the place at infinity of $K$;
    \item $K_\infty:=\F_q((1/\theta))$ is the $\infty$-adic completion of $K$;
    \item $\overline{K_\infty}$ is a fixed algebraic closure of $K_\infty$;
    \item $\C_\infty$ is the completion of $\overline{K_\infty}$;
    \item $v_\infty$ denotes the valuation over $K$ associated to $\infty$, normalized such that $v_\infty(\theta)=-1$, which is uniquely extended to $\C_\infty$;
    \item $F/K$ is a finite extension of degree $m$;
    \item $\OF$ is the integral closure of $A$ in $F$;
    \item $F_\infty := K_\infty\otimes_K F$, which is a $K_\infty$-vector space of dimension $m$;
    \item $L\subset \C_\infty$ is the perfection of $F$;
    \item $\tau:x\mapsto x^q$ will denote the Frobenius map in various settings.
    
    \end{itemize}
Let $M\subset \C_\infty$ be any field containing $K$. An \textit{Anderson $t$-module} (which we might also call Anderson $A$-module, or Anderson module for short) $E$ defined over $M$ (we denote $E/M$) of dimension $n$ is a morphism of $\F_q$-algebras
\begin{align*}
    E  : A & \rightarrow \Mat_{n\times n}(M)\{\tau\} \\
    \theta & \mapsto E_\theta = \partial_E(\theta) + \sum_{i=1}^{D} E_{i}\tau^i
\end{align*}
where $(\partial_E(\theta)-\theta\I_n)^n=0$. Note that $E$ is completely determined by the image of $\theta$. The independent term of $E$ induces a morphism of $\F_q$-algebras 
$$\partial_E: A \rightarrow \Mat_{n\times n}(M).$$

Consider the non-commutative ring $L[t,\sigma]$, where
$$tc=ct,\quad t\sigma=\sigma t,\quad \sigma c=c^{1/q}\sigma,\quad c\in L.$$
A \textit{dual $t$-motive} $N(L)$ over $L$ is a left $L[t,\sigma]$-module which is free and finitely generated over $L\{\sigma\}$ and such that there exists an integer $n$ with $(t-\theta)^n N(L)\subseteq\sigma N(L)$. A dual $t$-motive is said to be \textit{$A$-finite} if it is also free and finitely generated over $L[t]$. The \textit{rank} of an $A$-finite Anderson module $E$ is the rank of $N(L)$ as an $L[t]$-module. 

To every Anderson module $E/L$ of dimension $n$, we can associate a dual $t$-motive $N_E(L)$ over $L$ in the following way: we set $N_E(L)=\Mat_{1\times n}(L\{\sigma\})$, which is naturally a free module over $L\{\sigma\}$ of rank $n$, and with $t$-action given by
$$t\cdot h := h\left(\partial_E(\theta)^T +  \sum_{i=1}^D \sigma^iE_{i}^T\right)\qquad \text{for every }h\in N_{E}(L),$$
where $\cdot^T$ denotes matrix transposition. We note that $t\sigma \cdot h=\sigma t\cdot h$, since $\sigma$ acts on the left and $t$ on the right. Anderson showed that $E\mapsto N_E(L)$ is an equivalence of categories (\cite[Theorem 4.4.1]{bockle_rapid_2020}). 

In the following, we will consider Anderson modules $E$ defined over $\OF$, meaning that $E_\theta \in \Mat_{n \times n}(\OF)\{\tau\}$. Given an $\OF$-algebra $B$, an Anderson module $E/\OF$ gives us two new $A$-module structures over $B^n$.
\begin{enumerate}
    \item One given by the full Anderson module: $$a \cdot h = E_a(h),\quad  a\in A, h\in B^n,$$ which we denote by $E(B)$.
    \item One given by the first term: $$a \cdot h = \partial_E(a)h,\quad a\in A, h\in B^n,$$ which we denote by $\Lie_E(B)$.
\end{enumerate}

If $B=F_\infty$, then one can show that the $A$-action on $\LieE(F_\infty)$ can be extended continuously to a $K_\infty$-action. Moreover, we have that $\dim_{K_\infty}\LieE(F_\infty)=mn$ (see \cite[Lemma 1.7]{fang_special_2015}).

We recall now some additional facts about Anderson modules. We refer to \cite{goss_basic_1996}, \cite{demeslay_class_2022} and \cite[Chapter 7]{angles_arithmetic_2021}  for more details.

An Anderson module $E/\OF$ has two associated series, 
$$\exp_E = \sum_{k\geq 0} Q_k\tau^k,\quad \log_E=\sum_{k\geq 0}P_k\tau^k\,\in \I_n + \tau \Mat_{n\times n}(F)\{\{\tau\}\},$$ that verify the following properties:
\begin{align*}
    \exp_E \partial_E(a) = E_a \exp_E, \quad &\forall \, a\in A, \\
    \log_E E_a = \partial_E(a)\log_E, \quad &\forall a\in A,\\
    \exp_E\log_E = \log_E \exp_E &= \I_n .
\end{align*}
We recall that $\exp_E$ converges everywhere on $\C_\infty^n$, thus inducing a morphism of $A$-modules
$$\exp_E :\Lie_E(\C_\infty) \longrightarrow E(\C_\infty).$$
The series $\log_E$, on the other hand, does not converge on all of $\C_\infty^n$. We define the \textit{module of units} of $E$ as
$$U(E/\OF):=\{x\in \Lie_E(F_\infty) : \exp_E(x)\in E(\OF)\}.$$

Let $V$ be a sub-$K_\infty$-vector space of $\Lie_E(F_\infty)$. We say that a free $A$-module $M\subset V$ is an \textit{$A$-lattice} in $V$ if it admits an $A$-basis that is also a $K_\infty$-basis of $V$. For example, the $A$-modules $\Lie_A(\OF)$ and $U(E/\OF)$ are $A$-lattices in $\Lie_E(F_\infty)$. 

If $M = \oplus A b_i$, $M'=\oplus Ab_i'$ are two $A$-lattices in $V$, then their \textit{ratio of covolumes} $[M:M']_A$ is defined as 
$$[M:M']_A := \frac{\det_{(b_i)}(b_i')}{\sgn(\det_{(b_i)}(b_i'))} ,$$
where the sign function in the denominator is defined as follows. If $x\in K_\infty$ can be written as $x=\sum_{i\geq N}x_i \theta^{-i}$ with $x_i\in \F_q$ and $x_N\neq 0$, then $\sgn(x)=x_N$. One can check that the definition of ratio of covolumes does not depend on the choice of basis.

We also define the \textit{class module} of $E$ as
$$H(E/\OF):= \frac{E(F_\infty)}{\exp_E(\Lie_E(F_\infty)) + E(\OF)}.$$
It is a finitely generated and torsion $A$-module. We will be mainly interested in the generator of its Fitting ideal. If $M$ is a finitely generated torsion $A$-module, $\textit{i.e.}$ if
$$M \simeq \frac{A}{f_1 A} \times \ldots \times \frac{A}{f_l A},\quad \text{ for some } f_i \in A_+,$$
then the generator of its Fitting ideal is
$$[M]_A:=f_1f_2\ldots f_l.$$

\begin{theorem}[Class formula]
    The infinite product
    $$L(E/\OF):= \prod_{P}\frac{\left[\Lie_E\left(\frac{\OF}{P\OF}]\right)\right]_A}{\left[E\left(\frac{\OF}{P\OF}\right)\right]_A},$$
    where $P$ runs trough all monic irreducible elements of $A$, converges in $K_\infty^*$. Moreover, we have
    $$L(E/\OF)=[\Lie_E(\OF) : U(E/\OF)]_A \cdot [H(E/\OF)]_A.$$
\end{theorem}
\begin{proof}
    See \cite{fang_special_2015} or \cite{demeslay_class_2022}.
\end{proof}
We call $L(E/\OF)$ the \textit{special $L$-value} (or \textit{$L$-series}) attached to $E$.

In this paper, we will work with a very specific family of Anderson modules; we will consider $A$-finite Anderson modules $E/\OF$ such that $\sigma N_E(L) \subseteq (t-\theta)N_E(L)$. This condition implies that $r\leq n$, where $r$ is the rank of $E$ and $n$ the dimension of $E$. The most prominent example of such an Anderson module is $\Cn$, the $n$-th tensor power of the Carlitz module. In \cite{angles_special_2020}, Anglès, Ngo Dac and Tavares-Ribeiro proved the following reduced version of the class formula, which will be the starting point of our study.

\begin{theorem} [Reduced class formula]\label{th : class formula mod K}
    Let $E/\OF$ be an $A$-finite Anderson module of rank $r$ such that $\sigma  N_E(L)\subseteq (t-\theta)N_E(L)$. Then there exists a sub-$K_\infty$-vector space $W$ of $\LieE(F_\infty)$ of dimension $rm$ such that
    $U(E/\OF)\cap W$ and $\LieE(\OF)\cap W$ are $A$-lattices in $W$, and 
        $$L(E/\OF) = [\LieE(\OF)\cap W : U(E/\OF)\cap W]_{A}\cdot \alpha$$
        for some $\alpha \in K^*$.
\end{theorem}

Let
$$N_{E,0}(L):= \Mat_{1\times n}({L}) \subset N_E(L).$$
We can define an $A$-module structure over $N_{E,0}(L)$ in the following way:
$$a\cdot h :=(\partial_E(a)h^T)^T =h \partial_E(a)^T,\qquad h\in N_{E,0}(L),\, a\in A.$$
In order to ease the notation, we will denote $a\cdot h$ just by $\partial_E(a) h$. We have an isomorphism of $A$-modules
\begin{equation*}
    N_{E,0}(L) \xlongrightarrow{\sim} \LieE(L)
\end{equation*}
given by transposition. We will in the following identify both spaces.

Take $j\in \N$. If $\lambda \in L$, then we denote $\lambda^{(j)}=\lambda^{q^j}$, and if $h=(h_1,\ldots,h_n) \in N_{E,0}(L)$, then we denote $h^{(j)}=(h_1^{(j)},\ldots,h_n^{(j)})$.
\begin{defin} \label{def : inverse Frobenius}
For any $j\in \N$, define
$$\varphi_j : N_E(L)\longrightarrow N_{E,0}(L)$$
to be the unique map verifying the following properties:
\begin{enumerate}
    \item $\varphi_j(h_1+\lambda h_2)=\varphi_j(h_1)+\lambda^{(j)}\varphi_j(h_2)$ for all $h_1,h_2\in N_E(L)$ and $\lambda\in L$.
    \item $\varphi_j(t\cdot h)=\partial_E(\theta)\varphi_j(h)$ for all $h\in N_E(L)$.
    \item $\varphi_j(\sigma^kh)=\varphi_{j-k}(h)$ for all $j,k\in\N$ and $h\in N_E(L)$, where we put $\varphi_l=0$ if $l<0$.
    \item $\varphi_j(h)=h^{(j)}P_j^T$ for all $h\in N_{E,0}(L)$. Recall that $P_j$ is the $j$-th coefficient of $\log_E$.
\end{enumerate}
\end{defin}

\begin{defin}\label{def : W}
    Let $(\nu_1,\ldots,\nu_r)$ be an $L[t]$-basis of $N_E(L)$, and $(x_1,\ldots,x_m)$ a $K$-basis of $F$. The space $W$ of Theorem \ref{th : class formula mod K} is defined as follows. If $F/K$ is separable, then define
$$W:=\sum_{l=1}^m\sum_{j=1}^r\partial_E(K_\infty)x_l^{q^{k}}\varphi_k(\nu_j)$$
for some $k\gg 0$ (it is shown in \cite{angles_special_2020} that for large enough $k$ the definition does not depend on $k$). If $F/K$ is not separable, then let $F^{sep}$ be the separable closure of $K$ in $F$, and let $s = [F^{sep}:K]$. Choose $y_{s+1},\ldots,y_m$ such that $F=F^{\sep}\oplus \oplus_{l=s+1}^m K y_l$, and define
$$W: = \sum_{l=1}^s\sum_{j=1}^r\partial_E(K_\infty)x_l^{q^{k}}\varphi_k(\nu_j) + \sum_{l=s+1}^m\sum_{j=1}^r\partial_E(K_\infty)y_l\varphi_k(\nu_j).$$
As before, for large enough $k$ the definition does not depend on $k$.
\end{defin}

%%%%%%%%%%%%%%%%%%%%%%%%%%%%%%%%%%%%%%%%%%%%%%%%%%%%%%%%%%%%%%%%%%%%%%%%%%%%%%%%%%%%%%%%%%%%%%%%%%%%%%%%%%%%%%%%%%
%%%%%%%%%%%%%%%%%%%%%%%%%%%%%%%%%%%%%%%%%%%%%%%%%%%%%%%%%%%%%%%%%%%%%%%%%%%%%%%%%%%%%%%%%%%%%%%%%%%%%%%%%%%%%%%%%%
%%%%%%%%%%%%%%%%%%%%%%%%%%%%%%%%%%%%%%%%%%%%%%%%%%%%%%%%%%%%%%%%%%%%%%%%%%%%%%%%%%%%%%%%%%%%%%%%%%%%%%%%%%%%%%%%%%

\section{The $z$-twist}
In order to prove a $P$-adic version of Theorem \ref{th : class formula mod K}, we will make use of the $z$-deformation of an Anderson module. Let $z$ be a new variable, and let
\begin{itemize}
    \item $\widetilde{K}:=K(z), \quad \widetilde{K_\infty}:=\Fq(z)((1/\theta));$
    \item $\widetilde{F}:=F(z), \quad \widetilde{F_\infty}:=F\otimes_K \widetilde{K_\infty};$
    \item $\widetilde{L}:=L(z)$;
    \item $\tilde{A}:=\Fq(z)A$, \quad $\widetilde{\OF} :=\F_q(z)\OF$; 
    \item $\T(K_\infty) := \Fq[z]((1/\theta))=\left\{\sum_{i=0}^\infty a_iz^i \in K_\infty[[z]] : v_\infty(a_i)\rightarrow \infty\right\};$
    \item $\T_\infty(K) := \T(K_\infty) \cap K[[z]] = \left\{\sum_{i=0}^\infty a_iz^i \in K[[z]] : v_\infty(a_i)\rightarrow \infty\right\};$
    \item $\T(F_\infty) := F \otimes_K \T(K_\infty) = \left\{\sum_{i=0}^\infty a_iz^i \in F_\infty[[z]] : v_\infty(a_i)\rightarrow \infty\right\};$
    \item $\T_\infty(F) := \T(F_\infty) \cap F[[z]] = \left\{\sum_{i=0}^\infty a_iz^i \in F[[z]] : v_\infty(a_i)\rightarrow \infty\right\}.$
\end{itemize}
We also denote by $v_\infty$ the Gauss valuation on $\widetilde{K_\infty}$. This is the unique valuation on $\widetilde{K_\infty}$ such that if 
$$x=\sum_{i=0}^N a_i z^i \in K_\infty[z],\quad \text{ then } \quad v_\infty(x)=\min_i\{v_\infty(a_i)\}.$$
We note that $\widetilde{K_\infty}$ is complete with respect to this valuation.

Let $E/\OF$ be an Anderson module given by $E_\theta=\partial_E(\theta)+\sum_{i=1}^D E_{i}\tau^i$. The $z$-deformation of the Anderson module $E$ is the morphism of $\F_q(z)$-algebras
$$\widetilde{E}  : \widetilde{A} \rightarrow \Mat_{n\times n}(\widetilde{\OF})\{\tau\} $$
given by
$$\theta  \mapsto \widetilde{E}_\theta = \partial_E(\theta) + \sum_{i=1}^{D} E_{i}\tau^iz^i,$$
and extended $\F_q(z)$-linearly. We have that $\LieEt(\widetilde{F_\infty})$ is a $\widetilde{K_\infty}$-vector space of dimension $nm$, and $\LieEt(\widetilde{\OF})$ is an $\widetilde{A}$-lattice in $\LieEt(\widetilde{F_\infty})$. 

The twisted $L$-series is
$$L(\widetilde{E}/\widetilde{\OF}):=\prod_{P\in A}\frac{[\LieEt(\widetilde{\OF}/P\widetilde{\OF})]_{\widetilde{A}}}{[\widetilde{E}(\widetilde{\OF}/P\widetilde{\OF})]_{\widetilde{A}}} \in \T_\infty(K),$$
where here and in the following, $P$ runs through all monic irreducible elements of $A$.

The twisted exponential and logarithm series are defined as
$$\exp_{\widetilde{E}} := \sum_{k\geq 0} Q_k\tau^kz^k,\quad \log_{\widetilde{E}}:=\sum_{k\geq 0}P_k\tau^kz^k.$$
They are the unique series in $\I_n + \tau \Mat_{n\times n}(\widetilde{F})\{\{\tau\}\}$ verifying
\begin{align*}
    \exp_{\widetilde{E}} \partial_E(a) = \widetilde{E}_a \exp_{\widetilde{E}}, \quad &\forall \, a\in \widetilde{A}, \\
    \log_{\widetilde{E}} \widetilde{E}_a = \partial_E(a)\log_{\widetilde{E}}, \quad &\forall a\in \widetilde{A},\\
    \exp_{\widetilde{E}}\log_{\widetilde{E}} = \log_{\widetilde{E}} \exp_{\widetilde{E}} &= \I_n .
\end{align*}
We have that $\expEt$ converges on $\LieEt(\widetilde{F_\infty})$, while $\log_{\widetilde{E}}$ converges on a smaller subspace. 

We have two versions of the unit module in this case:
$$U(\widetilde{E}/\widetilde{\OF}):=\{x\in \Lie_E(\widetilde{F_\infty}):\exp_{\widetilde{E}}(x)\in \widetilde{E}(\widetilde{\OF})\},$$
$$U(\widetilde{E}/\OF[z]):=\{x\in \T(F_\infty) : \exp_{\widetilde{E}}(x)\in \widetilde{E}(\OF[z])\}.$$
It is known (\cite[Proposition 1]{angles_arithmetic_2021}) that $U(\widetilde{E}/\widetilde{\OF})$ is a $\widetilde{A}$-lattice in $\Lie_E(\widetilde{F_\infty})$, $U(\widetilde{E}/\OF[z])$ is a finitely-generated $A[z]$-module and
$$U(\widetilde{E}/\widetilde{\OF}) = \F_q(z)U(\widetilde{E}/\OF[z]).$$

To every twisted Anderson module $\widetilde{E}/\widetilde{L}$ of dimension $n$, we can associate a twisted dual $t$-motive $N_{\widetilde{E}}(\widetilde{L})$ over $\widetilde{L}$ in the following way. Consider the ring $\widetilde{L}\{\sigma\}$, where  $z$ commutes with everything. We set $N_{\widetilde{E}}(\widetilde{L})=\Mat_{1\times n}(\widetilde{L}\{\sigma\})$, which is naturally a free module over $\widetilde{L}\{\sigma\}$ of rank $n$. The $t$-action is in this case given by
$$t\cdot h := h\left(\partial_E(\theta)^T +  \sum_{i=1}^D z^i\sigma^iE_{i}^T\right)\qquad \text{for every }h\in N_{\widetilde{E}}(\widetilde{L}),$$
where as before $\cdot^T$ denotes matrix transposition. 

Let 
$$N_{\widetilde{E},0}(\widetilde{L}):= \Mat_{1\times n}({\widetilde{L}}) \subset N_{\widetilde{E}}(\widetilde{L}).$$
For a given $j\in \N$, if $x=\sum_{i=0}^N a_i z^i \in L[z]$, then define $x^{(j)}=\sum_{i=0}^N a_i^{q^j} z^i$, and if $x=f/g \in \widetilde{L}$ with $f,g \in L[z]$, then $x^{(j)}=f^{(j)}/g^{(j)}$. Lastly, if $h=(h_1,\ldots,h_n) \in N_{\widetilde{E},0}(\widetilde{L})$, we denote $h^{(j)}=(h_1^{(j)},\ldots,h_n^{(j)}$).
\begin{theorem}\label{th : class formula mod K tilde}
    Let $E/\OF$ be an $A$-finite Anderson module of rank $r$ such that $\sigma  N_E(L)\subseteq (t-\theta)N_E(L)$. Then there exists a sub-$\widetilde{K_\infty}$-vector space $\widetilde{W}$ of $\LieEt(\widetilde{F_\infty})$ of dimension $rm$ such that:
    \begin{enumerate}
        \item $\left(\widetilde{W} \cap \LieEt(\T(F_\infty))\right)|_{z=1}=W$. 
        \item $U(\widetilde{E}/\widetilde{O}_F)\cap \widetilde{W}$ and $\LieEt(\widetilde{O}_F)\cap \widetilde{W}$ are $\widetilde{A}$-lattices in $\widetilde{W}.$
        \item We have $$L(\widetilde{E}/\widetilde{O}_F) = [\LieEt(\widetilde{O}_F)\cap \widetilde{W} : U(\widetilde{E}/\widetilde{O}_F)\cap \widetilde{W}]_{\widetilde{A}}\cdot \widetilde{\alpha}$$
        for some $\alpha \in \widetilde{K}^*$. 
    \end{enumerate}
\end{theorem}
\begin{proof}
    The proof is essentially the same as that of Theorem \ref{th : class formula mod K}, applied to the twisted version of all our objects. Consider the map 
    $$\widetilde{\varphi}_j:N_{\widetilde{E}}(\widetilde{L})\rightarrow N_{\widetilde{E},0}(\widetilde{L})$$
    satisfying properties i), ii) and iii) of Definition \ref{def : inverse Frobenius}, and such that
    $$\widetilde{\varphi}_j(h)=z^jh^{(j)}P_j^T\quad \text{ for all}\quad h\in N_{{\widetilde{E}},0}(\widetilde{L}).$$
    We use the notations of Definition \ref{def : W}. Let $(v_1,\ldots,v_r)$ be a $\widetilde{L}[t]$-basis of $N_{\widetilde{E}}(\widetilde{L})$ such that $v_j|_{z=1}=\nu_j$ for all $j$. We set
    $$\widetilde{W} := \sum_{l=1}^m\sum_{j=1}^r\partial_E(\widetilde{K_\infty})x_l^{q^{k}}\widetilde{\varphi}_k(v_j)$$
    if $F/K$ is separable, and
    $$\widetilde{W}: = \sum_{l=1}^s\sum_{j=1}^r\partial_E(\widetilde{K_\infty})x_l^{q^{k}}\widetilde{\varphi}_k(v_j) + \sum_{l=s+1}^m\sum_{j=1}^r\partial_E(\widetilde{K_\infty})y_l\widetilde{\varphi}_k(v_j)$$
    if $F/K$ is not separable. The first point of the theorem follows directly from the definitions. The proof of points 2 and 3 is formally the same as in \cite{angles_special_2020}.
\end{proof}

\begin{lemma}\label{lem : ev(U cap W) is lattice}
    Both $\left(\LieEt(\OF[z])\cap \widetilde{W}\right)|_{z=1}$ and $\left(U(\widetilde{E}/\OF[z])\cap \widetilde{W}\right)|_{z=1}$ are $A$-lattices in $W$. Moreover, there exist
    \begin{itemize}
        \item an $\widetilde{A}$-basis $(\widetilde{u_1},\ldots,\widetilde{u_{rm}})$ of $U(\widetilde{E}/\widetilde{\OF})\cap \widetilde{W}$ such that $\widetilde{u_i}\in U(\widetilde{E}/\OF[z])$ for all $i$, and if $u_i=\widetilde{u_i}|_{z=1}$, then $\sum_i \partial_E(A) u_i$ is an $A$-lattice in $W$;
        \item an $\widetilde{A}$-basis $(\widetilde{b_1},\ldots,\widetilde{b_{rm}})$ of  $\LieEt(\widetilde{\OF})\cap \widetilde{W}$ such that $\widetilde{b_i}\in \LieE(\OF[z])$ for all $i$, and if $b_i=\widetilde{b_i}|_{z=1}$, then $\sum_i \partial_E(A) b_i$ is an $A$-lattice in $W$.
    \end{itemize}
\end{lemma}
\begin{proof}
    We give the proof for the unit module, the other case being very similar. Note that $\left(U(\widetilde{E}/\OF[z])\cap \widetilde{W}\right)|_{z=1}$ is a free $A$-module, since it is a submodule of the free $A$-module $U(E/\OF)\cap W$ and $A$ is a principal ideal domain. It remains to show that it is of maximal rank in $W$. 
    Let $\widetilde{u}_1,\ldots,\widetilde{u}_{rm}$ be an $\widetilde{A}$-basis of $U(\widetilde{E}/\widetilde{\OF})\cap \widetilde{W}$. Since 
    $$U(\widetilde{E}/\widetilde{\OF})\cap \widetilde{W}=\F_q(z)(U(\widetilde{E}/\OF[z])\cap \widetilde{W}),$$
    multiplying by an element of $\F_q(z)$ if necessary, we can assume that $\widetilde{u_i}\in U(\widetilde{E}/\OF[z])$ and $(z-1)\nmid \widetilde{u}_i$ for all $i$.
    
    Set $u_i:=\widetilde{u}_i|_{z=1}$. We claim that the elements $u_1,\ldots,u_{rm}$ are $K_\infty$-linearly independent. Indeed, suppose
    $$\sum_{i=1}^{rm}\partial_E(a_i)u_i=0$$
    for some $a_i\in K_\infty$. Then the point $\sum_{i=1}^{rm}\partial_E(a_i)\widetilde{u}_i$ vanishes at $z=1$. This implies $(z-1)|\partial_E(a_i)\widetilde{u}_i$ for all $i$, and therefore $(z-1)|a_i$. But $a_i\in K_\infty$, so $a_i=0$ for all $i$, as desired.
\end{proof}

\begin{prop} \label{prop : class formula mod K integral level}
    Let $E/\OF$ be an $A$-finite Anderson module of rank $r$ such that $\sigma  N_E(L)\subseteq (t-\theta)N_E(L)$. Take two $\widetilde{A}$-basis, $(\widetilde{b}_i)$ and $(\widetilde{u}_i)$, as in Lemma \ref{lem : ev(U cap W) is lattice}.
    Let $M=\Mat_{(\widetilde{b_i})}(\widetilde{u_i})\in \Mat_{rm\times rm}(\widetilde{K_\infty})$. Then 
    $$L(\widetilde{E}/\widetilde{O}_F)=\det (M)\cdot \widetilde{\alpha}$$
    for some $\alpha \in \widetilde{K}^*$. Moreover, neither $\widetilde{\alpha}$ nor $\det(M)$ have a zero or a pole at $z=1$.
\end{prop}

\begin{proof}
      Making use of Theorem \ref{th : class formula mod K tilde}, we obtain
    $$L(\widetilde{E}/\widetilde{O}_F)=\widetilde{\alpha} \det (M)$$
    for some $\widetilde{\alpha} \in \widetilde{K}^*$. 

We want to prove that $\det(M)$ has no zeros or poles at $z=1$. Let us extend both basis to $\widetilde{K_\infty}$-basis of $\LieEt(\widetilde{F_\infty})$, which we denote by $(\widetilde{u_1},\ldots,\widetilde{u_{nm}})$ and $(\widetilde{b_1},\ldots,\widetilde{b_{nm}})$.
     Denote by $\mathcal{B}$ a $K$-basis of $\LieE(F)$, which will then also be a $\T_\infty(K)$-basis of $\LieEt(\T_\infty(F))$. We have
$$M = \Mat_{(\widetilde{b_i})}(\widetilde{u_i}) = \Mat_{(\widetilde{b_i})}\mathcal{B}\cdot\Mat_{\mathcal{B}}(\widetilde{u_i}).$$
Since $\mathcal{B}$ is a $\T_\infty(K)$-basis of $\LieEt(\T_\infty(F))$, we have
$$\Mat_{\mathcal{B}}(\widetilde{u_i})\in \Mat_{nm\times nm}(\T_\infty(K)).$$ 
Moreover,
$$\Mat_{(\widetilde{b_i})}\mathcal{B}=\frac{1}{\det(\Mat_\mathcal{B}(\widetilde{b_i}))}\Mat_\mathcal{B}(\widetilde{b_i})^*\in\frac{1}{\delta(z)}\Mat_{nm\times nm}(\T_\infty(K)),$$
where $\delta(z)=\det(\Mat_\mathcal{B}(\widetilde{b_i}))\in K[z]$, and $*$ denotes here the adjugate matrix, \textit{i.e.} the transpose of the cofactor matrix. Note that $(\widetilde{b_i}|_{z=1})$ is a $K$-basis of $\LieE(F)$. Hence
$$\delta(1)=\det(\Mat_\mathcal{B}(b_i))\neq 0.$$
It follows that 
$$ \Mat_{(\widetilde{b_i})}(\widetilde{u_i}) \in \frac{1}{\delta(z)}\Mat_{nm\times nm}(\T_\infty(K)),$$
$$ M \in \frac{1}{\delta(z)}\Mat_{rm\times rm}(\T_\infty(K)).$$
Hence, $\det(M)$ does not have a pole at $z=1$. It does not have a zero either, since
$$\det(M)|_{z=1}\in K^* [\LieE(\OF)\cap W:U(E/\OF)\cap W]_{A}\in K_\infty^*.$$

Since both $L(\widetilde{E}/\widetilde{O}_F)$ and $\det(M)$ do not have a zero or a pole at $z=1$, the same must be true for $\widetilde{\alpha}.$
    
\end{proof}

%%%%%%%%%%%%%%%%%%%%%%%%%%%%%%%%%%%%%%%%%%%%%%%%%%%%%%%%%%%%%%%%%%%%%%%%%%%%%%%%%%%%%%%%%%%%%%%%%%%%%%%%%%%%%%%%%%
%%%%%%%%%%%%%%%%%%%%%%%%%%%%%%%%%%%%%%%%%%%%%%%%%%%%%%%%%%%%%%%%%%%%%%%%%%%%%%%%%%%%%%%%%%%%%%%%%%%%%%%%%%%%%%%%%%
%%%%%%%%%%%%%%%%%%%%%%%%%%%%%%%%%%%%%%%%%%%%%%%%%%%%%%%%%%%%%%%%%%%%%%%%%%%%%%%%%%%%%%%%%%%%%%%%%%%%%%%%%%%%%%%%%%

\section{A $P$-adic reduced class formula}
Let 
\begin{itemize}
    \item $P$ be a monic irreducible element of $A$;
    \item $K_P:=\F_{q^{\deg(P)}}((P))$ be the $P$-adic completion of $K$;
    \item $\widetilde{K_P} := \F_{q^{\deg(P)}}(z)((P))$;
    \item $F_P:=K_P\otimes_{\F_q}F$, \quad $\widetilde{F_P}:=\widetilde{K_P}\otimes_{\F_q}F$;
    \item $\T(K_P) := \F_{q^{\deg(P)}}[z]((P))=\left\{\sum_{i=0}^\infty a_iz^i \in K_P[[z]] : v_P(a_i)\rightarrow \infty\right\};$
    \item $\T_P(K) := \T(K_P) \cap K[[z]] = \left\{\sum_{i=0}^\infty a_iz^i \in K[[z]] : v_P(a_i)\rightarrow \infty\right\};$
    \item $\T(F_P) := F \otimes_K \T(K_P) = \left\{\sum_{i=0}^\infty a_iz^i \in F_P[[z]] : v_P(a_i)\rightarrow \infty\right\};$
    \item $\T_P(F) := \T(F_P) \cap F[[z]] = \left\{\sum_{i=0}^\infty a_iz^i \in F[[z]] : v_P(a_i)\rightarrow \infty\right\};$     
    \item $v_P$ be the $P$-adic valuation on $K_P$ normalized so that $v_P(P)=1$, which we extend by means of the Gauss valuation to $\widetilde{K_P}$.
\end{itemize}

If $E$ is an Anderson module, then the $P$-adic $L$-series associated to $E$ and $\widetilde{E}$ are
$$L_P(E/\OF):=\prod_{Q\neq P}\frac{[\LieE(\OF/Q\OF)]_A}{[E(\OF/Q\OF)]_A} \in K_P,$$
$$L_P(\widetilde{E}/\widetilde{\OF}):=\prod_{Q\neq P}\frac{[\LieEt(\widetilde{\OF}/Q\widetilde{\OF})]_{\widetilde{A}}}{[\widetilde{E}(\widetilde{\OF}/Q\widetilde{\OF})]_{\widetilde{A}}} \in \T_P(K),$$
where the products run over all monic irreducible elements $Q$ of $A$ different from $P$. We refer to \cite[Section 4]{lucas_P-adic_2026} for a proof of the convergence of these series. For any irreducible element $Q\in A_+$, let us denote the local factor at $Q$ by
$$Z_Q(E/\OF):=\frac{[\LieE(\OF/Q\OF)]_A}{[E(\OF/Q\OF)]_A} ,\quad Z_Q(\widetilde{E}/\widetilde{\OF}):=\frac{[\LieEt(\widetilde{\OF}/Q\widetilde{\OF})]_{\widetilde{A}}}{[\widetilde{E}(\widetilde{\OF}/Q\widetilde{\OF})]_{\widetilde{A}}}.$$
Note that formally in $K((z))$, we have
$$L(\widetilde{E}/\widetilde{\OF})=L_P(\widetilde{E}/\widetilde{\OF}) Z_P(\widetilde{E}/\widetilde{\OF}).$$

We can study the $P$-adic convergence of the logarithm and the exponential. Let us denote by $\exp_{\widetilde{E},P}$ (resp. $\log_{\widetilde{E},P}$) the series that is formally the same as $\exp_{\widetilde{E}}$ (resp. $\log_{\widetilde{E}}$), but where we now look at the convergence $P$-adically. It turns out that $\log_{\widetilde{E},P}$ converges on all elements of the Tate algebra of positive $P$-adic valuation, thus defining a function
$$\log_{\widetilde{E},P}:\{x\in \T(F_P)^n :v_P(x) > 0\} \rightarrow \LieEt(\T(F_P)).$$
The series $\exp_{\widetilde{E},P}$ has in this case a smaller convergence domain. We refer again to \cite[Section 4]{lucas_P-adic_2026} for more details. Following the aforementioned paper, we extend the convergence domain of the logarithm. Set $$g(z)=g_P(z):=[\widetilde{E}(\widetilde{\OF}/P\widetilde{\OF})]_{\widetilde{A}}\in A[z],$$
so that $$g(1)=[E(\OF/P\OF)]_A\in A\setminus\{0\}.$$
Take $s\in \N$ such that $q^s\geq n$. If $h(z)=g(z)^{q^s}$, we have $\partial_E(h(z))=h(z) I_n$. Using the fact that the Fitting ideal is contained in the annihilator, we get 
$$\widetilde{E}_{h(z)}(\widetilde{\OF}^n)\in P\widetilde{\OF}^n.$$
Hence we have a well-defined function
\begin{align*}
    \Log_{\widetilde{E},P}:\{x\in \T(F_P)^n :v_P(x) \geq 0\} &\rightarrow \frac{1}{h(z)}\LieEt(\T(F_P)),\\
    x&\mapsto \frac{1}{h(z)}\log_{\widetilde{E},P}(\widetilde{E}_{h(z)}(x)).
\end{align*}
We can of course evaluate at $z=1$, defining a function
\begin{align*}
    \Log_{E,P}:\{x\in F_P^n :v_P(x) \geq 0\} &\rightarrow \LieE(F_P),\\
    x&\mapsto \frac{1}{h(1)}\log_{E,P}(E_{h(1)}(x)).
\end{align*}
These maps extend the usual $P$-adic logarithm, and satisfy the expected properties (see \cite[Lemma 4.17]{lucas_P-adic_2026}). In particular, formally in $\Mat_{n\times n}(\widetilde{F})\{\{\tau\}\}$ we have
$$\Log_{\widetilde{E},P}\exp_{\widetilde{E}} = \exp_{\widetilde{E}}\Log_{\widetilde{E},P} = \I_n.$$

\begin{prop}\label{prop : kernel Log torsion points}
    Let $x\in \OF^n$. Then 
    $$\Log_{E,P}(x)=0$$ 
    if and only if $x$ is a torsion point of $E(\OF)$.
\end{prop}
\begin{proof}
    See \cite[Proposition 4.20]{lucas_P-adic_2026}.
\end{proof}

The $A$-action on $\LieE(F_P)$ can be extended to a $K_P$-action, thus making $\LieE(F_P)$ into a $K_P$-vector space of dimension $nm$ (see \cite[Proposition 4.1]{lucas_P-adic_2026}). We also have that $\LieE(\OF)$ is an $A$-lattice in $\LieE(F_P)$. Moreover, $\LieEt(\widetilde{F_P)}$ is a $\widetilde{K_P}$-vector space of dimension $nm$, and $\LieEt(\widetilde{\OF})$ is an $\widetilde{A}$-lattice in $\LieEt(\widetilde{F_P})$.

\begin{defin}
    Let $V \subseteq \LieE(F_\infty)$ be a sub-$K_\infty$-vector space of dimension $s$, and let
    $$\Lambda=\oplus_{i=1}^{s} \partial_E(A)b_i \subseteq \Lie_E(F)\cap V, \quad \Lambda' =\oplus_{i=1}^{s} \partial_E(A)u_i \subseteq U(E/\OF)\cap V$$
    be two $A$-lattices in $V$. Define their $P$-adic ratio of covolumes as
    $$[\Lambda : \Lambda']_{A,P} := \frac{\det_{(b_i)}(\Log_{E,P}(\exp_{E}(u_i)))}{\sgn(\det_{(b_i)}(u_i))}.$$
    Note that the determinant is computed over the $K_P$-vector space $\LieE(F_P)$. As before, the definition does not depend on the choice of basis. The same definition can also be made with the twist by $z$.
\end{defin}

We now prove a $P$-adic version of Proposition $\ref{prop : class formula mod K integral level}$.
\begin{prop} \label{prop : P-adic class formula integral level}
    Let $E/\OF$ be an $A$-finite Anderson module such that $\sigma  N_E(L)\subseteq (t-\theta)N_E(L)$. Take two $\widetilde{A}$-basis, $(\widetilde{b}_i)$ and $(\widetilde{u}_i)$, as in Lemma \ref{lem : ev(U cap W) is lattice}.
    Let 
    $$M=\Mat_{(\widetilde{b_i})}(\Log_{\widetilde{E},P}(\exp_{\widetilde{E}}(\widetilde{u_i})))\in \Mat_{rm\times rm}(\widetilde{K_P}).$$ 
    Then 
    $$L_P(\widetilde{E}/\widetilde{O}_F)=\det (M)\cdot \widetilde{\alpha}$$
    for some $\widetilde{\alpha} \in \widetilde{K}^*$ that does not have a zero or a pole at $z=1$. Moreover, $\det(M)$ does not have a pole at $z=1$.
\end{prop}
\begin{proof}
Note that formally in $F((z))^n$, we have 
$$\widetilde{u_i} = \Log_{\tilde{E},P}(\exp_{\tilde{E}}(\widetilde{u_i})).$$
Hence, by Proposition \ref{prop : class formula mod K integral level}, formally in $K((z))$ we have
$$L(\widetilde{E}/\widetilde{\OF})=\det\left(M)\right)\cdot \widetilde{\beta}$$
for some $\widetilde{\beta}$ that does not have a zero or a pole at $z=1$. Hence
$$L_P(\widetilde{E}/\widetilde{\OF})= \det(M)\cdot\underbrace{\left(Z_P(\widetilde{E}/\widetilde{\OF})^{-1}\widetilde{\beta}\right)}_{\widetilde{\alpha}}.$$
Since the local factor $Z_P(\widetilde{E}/\widetilde{\OF}$ does not have a zero or a pole at $z=1$, the same is true for $\widetilde{\alpha}$.

We recover the notation of the proof of Proposition \ref{prop : class formula mod K integral level}. Denote by $\mathcal{B}$ a $\partial_E(K)$-basis of $\LieE(F)$, which will then also be a $\T_\infty(K)$-basis of $\LieEt(\T_\infty(F))$ and a $\T_P(K)$-basis of $\LieEt(\T_P(F))$. Let $\delta(z)=\det(\Mat_\mathcal{B}(\widetilde{b_i}))\in K[z]$, and observe that $\delta(1)\neq 0$. Recall that $\widetilde{u_i}\in \Log_{\tilde{E},P}(\OF[z])\in \frac{1}{h(z)}\LieEt(\T_P(F))$, with $h(1)\neq 0$. We have
$$\Mat_{(\widetilde{b_i})}(\widetilde{u_i}) = \Mat_{(\widetilde{b_i})}\mathcal{B}\cdot\Mat_{\mathcal{B}}(\widetilde{u_i}) \in \frac{1}{h(z)\delta(z)}\Mat_{nm\times nm}(\T_P(K)),$$
$$M \in \frac{1}{h(z)\delta(z)}\Mat_{rm\times rm}(\T_P(K)).$$
Hence $\det(M)$ does not have a $P$-adic pole at $z=1$.
\end{proof}
  
\begin{theorem}[$P$-adic reduced class formula]\label{th : P-adic formula W}
    Let $E/\OF$ be an $A$-finite Anderson module such that $\sigma  N_E(L)\subseteq (t-\theta)N_E(L)$. Over $\widetilde{K_P}$, we have the equality
    $$L_P(\widetilde{E}/\widetilde{\OF})= [\LieEt(\widetilde{\OF})\cap \widetilde{W} : U(\widetilde{E}/\widetilde{\OF})\cap \widetilde{W}]_{\widetilde{A},P} \cdot \widetilde{\alpha}$$
    for some $\widetilde{\alpha}\in \widetilde{K}^*$. Over $K_P$, we have
    $$L_P(E/\OF)= [\LieE(\OF)\cap W : U(E/\OF)\cap W]_{A,P}\cdot \alpha$$
    for some $\alpha\in K^*$.
\end{theorem}
\begin{proof}
 Take $\widetilde{u_1},\ldots,\widetilde{u_{rm}}\in U(\widetilde{E}/\OF[z]) \cap \widetilde{W}$ and $\widetilde{b_1},\ldots,\widetilde{b_{rm}}\in \Lie_E(\widetilde{\OF}[z])\cap \widetilde{W}$ as in Lemma \ref{lem : ev(U cap W) is lattice}. By the preceding proposition,
 $$L_P(\widetilde{E}/\widetilde{O}_F)=\det (M)\cdot \widetilde{\beta}$$
for some $\widetilde{\beta} \in \widetilde{K}^*$. It follows that
$$L_P(\widetilde{E}/\widetilde{\OF})= [\LieEt(\widetilde{\OF})\cap \widetilde{W} : U(\widetilde{E}/\widetilde{\OF})\cap \widetilde{W}]_{\widetilde{A},P}\cdot\overbrace{\left(\sgn(\det(M))\widetilde{\beta}\right)}^{\widetilde{\alpha}}.$$
This proves the first assertion.

Evaluating the equation of Proposition \ref{prop : P-adic class formula integral level} \ at $z=1$, we obtain
$$L_P(E/\OF)= \det_{(b_i)}(\Log_{E,P}(\exp_{E}(u_i))) \cdot\beta ,$$ 
for some $\beta \in K^*$.
Suppose 
$$\LieE(A) \cap W = \oplus_{i=1}^{rm} \partial_E(A)c_i,\quad U(E/\OF) \cap W = \oplus_{i=1}^{rm} \partial_E(A) w_i.$$
Note that
\begin{multline*}
    \Mat_{(b_i)}\left(\Log_{E,P}(\exp_E(u_i))\right) = \\ 
    \Mat_{(b_i)}\left(c_i\right) \cdot \Mat_{(c_i)}\left(\Log_{E,P}(\exp_E(w_i))\right) \cdot \Mat_{(w_i)}\left(u_i\right).
\end{multline*}
Setting
$$\alpha = \beta \cdot \det_{(b_i)}\left(c_i\right) \cdot \det_{(w_i)}\left(u_i\right) \cdot \sgn(\det_{(c_i)}(w_i)) \in K^*,$$
we obtain the desired result.
\end{proof}

\begin{cor}\label{cor : vanishing P-adic L}
    Let $E/A$ be an $A$-finite Anderson module of rank $1$ such that $\sigma  N_E(L)\subseteq (t-\theta)N_E(L)$. Then
    $$L_P(E/A)=0$$
    if and only if
    $\exp_E:\LieE(K_\infty)\rightarrow E(K_\infty)$
    is not injective.
\end{cor}
\begin{proof}
    The implication from right to left is proven in greater generality in \cite[Theorem 5.2]{lucas_P-adic_2026} using the general $P$-adic class formula. 
    
    For the other implication, suppose $L_P(E/A)=0$. Since $E$ is of rank $1$, we have that $U(E/A)\cap W$ is a free $A$-module of rank $1$, so we can write $U(E/A)\cap W = \partial_E(A)u$ for some nonzero $u\in \Lie_E(K_\infty)$. Then by the preceding theorem, the fact that $L_P(E/A)=0$ implies that
    $$\Log_{E,P}(\exp_E(u))=0.$$
    By Proposition \ref{prop : kernel Log torsion points}, this in turn implies that 
    $$\exp_E(u)\in E(A)_{tors}.$$
    Therefore, for some nonzero $a\in A$ we have
    $$0=E_a(\exp_E(u))=\exp_E(\partial_E(a)u),$$
    which means that the exponential is not injective.
\end{proof}

\begin{remark}\label{rem : 1}
    The proof of Corollary \ref{cor : vanishing P-adic L} tells us that for such Anderson modules $E$, if $L_P(E/A)=0$, then $\exp_E(u)\in E(A)_{tors}$ for all $u\in U(E/A)\cap W$.
\end{remark}

%%%%%%%%%%%%%%%%%%%%%%%%%%%%%%%%%%%%%%%%%%%%%%%%%%%%%%%%%%%%%%%%%%%%%%%%%%%%%%%%%%%%%%%%%%%%%%%%%%%%%%%%%%%%%%%%%%
%%%%%%%%%%%%%%%%%%%%%%%%%%%%%%%%%%%%%%%%%%%%%%%%%%%%%%%%%%%%%%%%%%%%%%%%%%%%%%%%%%%%%%%%%%%%%%%%%%%%%%%%%%%%%%%%%%
%%%%%%%%%%%%%%%%%%%%%%%%%%%%%%%%%%%%%%%%%%%%%%%%%%%%%%%%%%%%%%%%%%%%%%%%%%%%%%%%%%%%%%%%%%%%%%%%%%%%%%%%%%%%%%%%%%

\section{Order of vanishing of $L_P(E/A)$} \label{sec : ord van L_P}
Let $v$ be either the $\infty$-adic or a $P$-adic place. Let $f\in \T(K_v)$. The order of vanishing at $z=1$ of $f$ is the largest integer $k_0\geq 0$ such that $(z-1)^{k_0} | f$.

Take $k\in \N$. The $k$-th hyperderivative with respect to $z$ is the $K_v$-linear map
$$\mathcal{D}_z^{(k)}:K_v[[z]] \longrightarrow K_v[[z]]$$
given by
$$\D_z^{(k)}(z^m)=\binom{m}{k}z^{m-k}, \quad m\geq 0,$$
where we say that $\binom{m}{k}=0$ if $k>m$.
Note that
$$\D_z:=\D_z^{(1)}$$
is just the standard derivative $d/dz$. We refer to \cite{jeong_calculus_2011} for more details about hyperderivatives. The main property we will use is the product formula
$$\D_z^{(k)}(fg) = \sum_{i+j=k}\D_z^{(i)}(f)\cdot\D_z^{(j)}(g),\qquad f,g\in K_v[[z]].$$
If $f\in \T_v(K)$, then the Taylor expansion
$$f=\sum_{k\geq 0} (\D_z^{(k)}f)|_{z=1}(z-1)^k$$
tells us that the order of vanishing of $f$ at $z=1$ is $k_0$ if and only if $k_0$ is the largest nonnegative integer such that $(\D_z^{(k_0)}f)|_{z=1}\neq 0$. 

\begin{defin}
    Let $E/\OF$ be an Anderson module. We define the order of vanishing of $L_P(E/\OF)$ as the order of vanishing at $z=1$ of $L_P(\widetilde{E}/\widetilde{\OF})$.
\end{defin}

We will make use of hyperderivatives to study the order of vanishing of $P$-adic $L$-series. For any $k\in \N$, set
$$L_P^{(k)}(E/A) := \left(\D_z^{(k)} L_P(\widetilde{E}/\widetilde{A})\right)|_{z=1}.$$

In this section, we will work with Anderson modules such that the space $W$ is of dimension $1$. More concretely, we will consider modules $E/A$ of rank $1$. In this case, the $P$-adic ratio of covolumes of Theorem \ref{th : P-adic formula W} is much easier to compute, since it is given by the single entry of a $1\times 1$ matrix.

Note that $\D_z^{(k)}$ can be naturally extended to a map
$$\D_z^{(k)}:\Mat_{n\times n}(K_v)[[z]] \rightarrow \Mat_{n\times n}(K_v)[[z]].$$
\begin{lemma} \label{lem : d(g)'=0 iff g'=0}
    Let $v$ denote either the $\infty$-adic or a $P$-adic valuation on $K$, and  let $E/A$ be an Anderson module. Then
    $$\D_z^{(k)}\partial_E(\widetilde{g})=\partial_E\left(\D_z^{(k)}\widetilde{g}\right)$$
    for any $\widetilde{g}\in \widetilde{K_v}$ and any $k\in \N$.
\end{lemma}
\begin{proof}
    If $\widetilde{g}=\sum_{i=0}^lg_iz^i \in K[z]$, then
    $$\partial_E\left(\D_z^{(k)}\widetilde{g}\right) = \partial_E\left(\sum_{i=1}^l\binom{i}{k}g_iz^{i-k}\right) = \sum_{i=1}^l\binom{i}{k}\partial_E(g_i)z^{i-k} = \D_z^{(k)}\partial_E(\widetilde{g}).$$
    Using the fact that $\partial_E$ is also multiplicative, we see that 
    $$\partial_E\left(\D_z^{(k)}\frac{\widetilde{f}}{\widetilde{g}}\right)  = \D_z^{(k)}\partial_E\left(\frac{\widetilde{f}}{\widetilde{g}}\right)$$
    for $\widetilde{f},\widetilde{g} \in K[z]$. The fact that
    $$\D_z^{(k)}\partial_E(\widetilde{g})=\partial_E\left(\D_z^{(k)}\widetilde{g}\right)$$
    for any $\widetilde{g}\in \widetilde{K_v}$ is a consequence of the $v$-adic continuity of $\partial_E$ and $\D_z^{(k)}$. 
\end{proof}

\begin{prop}\label{prop : derivative P-adic zero iff torsion}
    Let $E/A$ be an $A$-finite Anderson module of rank $1$ such that $\sigma  N_E(L)\subseteq (t-\theta)N_E(L)$. Suppose $L_P(E/A)=0$. Take $\widetilde{w}\in U(\widetilde{E}/A[z])\cap\widetilde{W}$ such that $w:=\widetilde{w}|_{z=1}\in U(E/A)\cap W$ is nonzero. Then $\exp_E(w)$ is a torsion point for $E$ by Remark \ref{rem : 1}; take $a\in A\setminus\{0\}$ such that $E_a(\exp_{E}(w))=0$. Then 
    $$L_P^{(1)}(E/A)=0$$
    if and only if $$\left(\D_z(\widetilde{E_a}(\exp_{\widetilde{E}}(\widetilde{w})))\right)|_{z=1}$$
    is a torsion point for $E$. 
\end{prop}
\begin{proof}
    Take $\widetilde{b}\in\LieEt(A[z])$ such that $\LieEt(\widetilde{A})\cap \widetilde{W}=\partial_E(\widetilde{A})\widetilde{b}$ and $b:=\widetilde{b}|_{z=1}\neq 0$. Recall that by the functional equation of the exponential,
    $$\widetilde{E_a}(\exp_{\widetilde{E}}(\widetilde{w})) = \exp_{\widetilde{E}}(\partial_E(a)\widetilde{w}),$$
    and $\partial_E(a)\widetilde{w} \in U(\widetilde{E}/A[z])\cap\widetilde{W}$. By Proposition \ref{prop : P-adic class formula integral level}, if we write
    \begin{equation}\label{eq : aw = gb}
        \Log_{\widetilde{E},P}(\widetilde{E_a}(\exp_{\widetilde{E}}(\widetilde{w})))=\partial_E(\widetilde{g}) \widetilde{b},
    \end{equation}
    for some $\widetilde{g}\in \widetilde{K_P}$ that converges $P$-adically at $z=1$, then
    \begin{equation} \label{eq : LP=ga}
    L_P(\widetilde{E}/\widetilde{A})=\widetilde{g}\cdot\widetilde{\alpha}
    \end{equation}
    for some $\widetilde{\alpha}\in \widetilde{K}$ that does not have a pole or a zero at $z=1$.

    \begin{claim}
        We claim that 
    $$L_P^{(1)}(E/A) = 0 \iff \left(\D_z\left(\Log_{\widetilde{E},P}(\widetilde{E_a}(\exp_{\widetilde{E}}(\widetilde{w})))\right)\right)|_{z=1}=0.$$
    \end{claim}
    \begin{proof}
        Write $\alpha=\widetilde{\alpha}|_{z=1}$. Since $L_P(E/A)=0$, by Equation \eqref{eq : LP=ga} we must have $\widetilde{g}|_{z=1}=0$. Taking derivatives  in \eqref{eq : LP=ga} yields 
    $$\D_z  L_P(\widetilde{E}/\widetilde{A}) = \D_z\widetilde{g}\cdot \widetilde{\alpha} + \widetilde{g}\cdot  \D_z\widetilde{\alpha};$$
    hence
    $$L_P^{(1)}(E/A)= \left(\D_z\widetilde{g}\right)|_{z=1}\cdot \alpha.$$
    Therefore, we have
    $$L_P^{(1)}(E/A) = 0 \iff \left(\D_z\widetilde{g}\right)|_{z=1}=0.$$
    By Equation \eqref{eq : aw = gb}, we obtain
    $$\left(\D_z\left(\Log_{\widetilde{E},P}(\widetilde{E_a}(\exp_{\widetilde{E}}(\widetilde{w})))\right)\right)|_{z=1} =\left(\D_z\partial_E(\widetilde{g})\right)|_{z=1}\cdot b ,$$
    and by Lemma \ref{lem : d(g)'=0 iff g'=0}, 
    $$\left(\D_z\left(\Log_{\widetilde{E},P}(\widetilde{E_a}(\exp_{\widetilde{E}}(\widetilde{w})))\right)\right)|_{z=1}=0 \iff \left(\D_z\widetilde{g}\right)|_{z=1}=0 ,$$ 
    as desired.
    \end{proof}
    Next, note that
    \begin{align*}
        \D_z\left(\Log_{\widetilde{E},P}(\widetilde{E_a}(\exp_{\widetilde{E}}(\widetilde{w})))\right) &= \\ (\D_z\Log_{\widetilde{E},P})(\widetilde{E_a}(\exp_{\widetilde{E}}(\widetilde{w})&)) + \Log_{\widetilde{E},P}(\D_z\widetilde{E_a}(\exp_{\widetilde{E}}(\widetilde{w}))).
    \end{align*}
    Thus
    $$\left(\D_z\left(\Log_{\widetilde{E},P}(\widetilde{E_a}(\exp_{\widetilde{E}}(\widetilde{w})))\right)\right)|_{z=1} = \Log_{E,P}\left(\left(\D_z\widetilde{E_a}(\exp_{\widetilde{E}}(\widetilde{w}))\right)|_{z=1}\right).$$
    By Proposition \ref{prop : kernel Log torsion points}, this quantity vanishes if and only if 
    $$\left(\D_z\widetilde{E_a}(\exp_{\widetilde{E}}(\widetilde{w}))\right)|_{z=1}$$
    is a torsion point. This concludes the proof.
\end{proof}

\begin{theorem}\label{th : vanishing does not depend P}
     Let $E/A$ be an $A$-finite Anderson module of rank $1$ such that $\sigma  N_E(L)\subseteq (t-\theta)N_E(L)$. The order of vanishing of $L_P(\widetilde{E}/\widetilde{A})$ at $z=1$ does not depend on $P$.
\end{theorem}
\begin{proof}
    By Corollary \ref{cor : vanishing P-adic L} the vanishing of $L_P(E/A)$ does not depend on $P$, and by Proposition \ref{prop : derivative P-adic zero iff torsion} the vanishing of $L_P^{(1)}(E/A)$ does not depend on $P$ either. We can repeat the method used in the cited proposition to obtain the result for higher hyperderivatives. For the sake of readability, we present the argument only for the second hyperderivative; the general case follows by the same reasoning.
    
    We keep the notations of the proof of Proposition \ref{prop : derivative P-adic zero iff torsion}. Suppose that 
    $$L_P^{(1)}(E/A)=0,$$ 
    which means that there exists a nonzero $b\in A$ such that 
    $$E_b\left(\left(\D_z\widetilde{E_a}(\exp_{\widetilde{E}}(\widetilde{w}))\right)|_{z=1}\right)=0.$$
    By a similar computation as in the previous proposition,
    $$L_P^{(2)}(E/A)=0 \iff \left(\D_z^{(2)}\left(\Log_{\widetilde{E},P}(\widetilde{E_{ba}}(\exp_{\widetilde{E}}(\widetilde{w})))\right)\right)|_{z=1}=0.$$
    Now
    \begin{align*}
        \D_z^{(2)}&\left(\Log_{\widetilde{E},P}(\widetilde{E}_{ba}(\exp_{\widetilde{E}}(\widetilde{w})))\right)=\left(\D_z^{(2)}\Log_{\widetilde{E},P}\right)(\widetilde{E}_{ba}(\exp_{\widetilde{E}}(\widetilde{w}))) \\
       & + 2\left(\D_z\Log_{\widetilde{E},P}\right)\left(\widetilde{E_b}\left(\D_z\widetilde{E}_{a}(\exp_{\widetilde{E}}(\widetilde{w}))\right) + \left(\D_z\widetilde{E}_{b}\right)(\widetilde{E}_{a}(\exp_{\widetilde{E}}(\widetilde{w})))\right) \\
        &+ \Log_{\widetilde{E},P}\left(\D_z^{(2)}(\widetilde{E}_{ba}(\exp_{\widetilde{E}}(\widetilde{w})))\right).
    \end{align*}
    When we evaluate at $z=1$, everything vanishes except the last term:
    $$ \left(\D_z^{(2)}\left(\Log_{\widetilde{E},P}(\widetilde{E_{ba}}(\exp_{\widetilde{E}}(\widetilde{w})))\right)\right)|_{z=1}= \Log_{E,P}\left(\left(\D_z^{(2)}(\widetilde{E}_{ba}(\exp_{\widetilde{E}}(\widetilde{w})))\right)|_{z=1}\right).$$
    Thus, $L_P^{(2)}(E/A) = 0$ if and only if
    $$\left(\D_z^{(2)}(\widetilde{E}_{ba}(\exp_{\widetilde{E}}(\widetilde{w})))\right)|_{z=1}$$
    is a torsion point for $E$. In particular, the vanishing of $L_P^{(2)}(E/A)$ does not depend on $P$. 
\end{proof}

%%%%%%%%%%%%%%%%%%%%%%%%%%%%%%%%%%%%%%%%%%%%%%%%%%%%%%%%%%%%%%%%%%%%%%%%%%%%%%%%%%%%%%%%%%%%%%%%%%%%%%%%%%%%%%%%%%
%%%%%%%%%%%%%%%%%%%%%%%%%%%%%%%%%%%%%%%%%%%%%%%%%%%%%%%%%%%%%%%%%%%%%%%%%%%%%%%%%%%%%%%%%%%%%%%%%%%%%%%%%%%%%%%%%%
%%%%%%%%%%%%%%%%%%%%%%%%%%%%%%%%%%%%%%%%%%%%%%%%%%%%%%%%%%%%%%%%%%%%%%%%%%%%%%%%%%%%%%%%%%%%%%%%%%%%%%%%%%%%%%%%%%

\section{$P$-adic Carlitz zeta values}
\subsection{Order of vanishing}
Define the \textit{Carlitz zeta value} at $n$ as
$$\zeta(n):= \sum_{a\in A_{+}} \frac{1}{a^n}\in K_\infty,$$
and the $z$-twisted Carlitz zeta value as
$$\zeta(n,z):= \sum_{d\geq 0}\sum_{a\in A_{+,d}} \frac{1}{a^n}z^d \in \T_\infty(K).$$

Let $P\in A$ be a monic irreducible element. Define the $P$-adic Carlitz zeta value as
$$\zeta_P(n):= \sum_{d\geq 0}\sum_{\substack{a\in A_{+,d}\\ P\nmid a}} \frac{1}{a^n} \in K_P,$$
where we are now looking at the $P$-adic topology. For convergence reasons, it is important in this case to first sum by degrees (see \cite[Section 5.5(b)]{thakur_function_2004}). Let also
$$\zeta_P(n,z):= \sum_{d\geq 0}\sum_{\substack{a\in A_{+,d}\\ P\nmid a}} \frac{1}{a^n}z^d\in \T_P(K).$$
Set
$$\zeta_P^{(1)}(n):=\left(\frac{d}{dz}\zeta_A(n,z)_P\right)|_{z=1}.$$

The $n$-th tensor power of the Carlitz module is the $\F_q$-algebra morphism
\begin{align*}
    \Cn  : A &\rightarrow \Mat_{n\times n}(K)\{\tau\} \\
    \theta & \mapsto \Cn_\theta = \partial_{\Cn}(\theta) + N\tau,
\end{align*}
where 
$$\partial_{\Cn}(\theta) = \begin{pmatrix}
    \theta & 1  & \cdots & 0 \\
    0 & \ddots & \ddots & \vdots \\
    \vdots & \ddots & \ddots & 1\\
    0 & \cdots & 0 & \theta
\end{pmatrix}\quad \text{ and } \quad N = \begin{pmatrix}
    0 & \cdots&\cdots& 0\\
    \vdots &&&\vdots \\
    0 &&& \vdots\\
    1 &  0 &\cdots& 0
\end{pmatrix}.$$
In the remaining of this section, we will denote $\partial=\partial_{\Cn}$.

The $L$-series attached to $\Cn$ is exactly the Carlitz zeta value at $n$. Indeed, by \cite[Section 4.2]{demeslay_class_2022} we have that 
$$[\Lie_{\Cn}(\widetilde{A}/Q\widetilde{A})]_{\widetilde{A}} = Q^n \quad \text{ and } \quad [\Cn(\widetilde{A}/Q\widetilde{A})]_{\widetilde{A}} = Q^n-z^{\deg_\theta(Q)}.$$
Therefore
\begin{align*}
    L_P(\widetilde{\Cn}/\widetilde{A})=\prod_{Q\neq P} \frac{[\Lie_{\Cn}(\widetilde{A}/Q\widetilde{A})]_{\widetilde{A}}}{[\Cn(\widetilde{A}/Q\widetilde{A})]_{\widetilde{A}}} = \prod_{Q\neq P} \frac{1}{1-z^{\deg_\theta(Q)}Q^{-n}} = \zeta_P(n,z),
\end{align*}
and hence $L_P(\Cn / A)=\zeta_P(n)$.

Let
$$e_n = \begin{pmatrix}
    0 \\
    \vdots \\
    0 \\
    1
\end{pmatrix}\in \Mat_{n\times 1}(K).$$
Recall that $L\subset \C_\infty$ is the perfection of $K$. One can check that $N_{\Cn}(L)=L[t]e_n^T$, and so $\Cn$ is an $A$-finite Anderson module of rank $1$. We also clearly have $\sigma N_{\Cn}(L)\subseteq (t-\theta)N_{\Cn}(L)$. Hence we can apply the results of Section \ref{sec : ord van L_P} to $\Cn$.

The \textit{Carlitz period} is defined as
$$\widetilde{\pi} = \theta\sqrt[q-1]{-\theta}\prod_{i=1}^\infty (1-\theta^{1-q^i})^{-1} \in \overline{K_\infty},$$
where $\sqrt[q-1]{-\theta}$ is any fixed $q-1$-th root of $-\theta$. Denote by 
$$\iota((x_1,\ldots,x_n)^T)=x_n$$
the projection onto the last component. We  have that
$$\ker(\exp_{\Cn}:\Lie_{\Cn}(\overline{K_\infty})\rightarrow \Cn(\overline{K_\infty})) = \partial(A)\pi_n,$$
for some $\pi_n\in \Lie_{\Cn}(\overline{K_\infty})$ such that $\iota(\pi_n)=\widetilde{\pi}^n$. Moreover, $\exp_{\Cn}$ is not injective on $\Lie_{\Cn}(K_\infty)$ if and only if $q-1\mid n$. As a consequence of Corollary \ref{cor : vanishing P-adic L}, we recover the following well-known fact:
$$\zeta_P(n)=0 \iff q-1 \mid n.$$
\begin{theorem}\label{th : order vanishing 1 riemann hyp}
   If $q-1\mid n$, then 
    $$\ord_{z=1}\zeta_P(n,z)=1.$$
\end{theorem}
\begin{proof}
    By Theorem \ref{th : vanishing does not depend P}, it is enough to check that $\zeta_P(n,z)$ has a simple zero at $z=1$ for $P=\theta$. But this is a consequence of the $\theta$-adic version of the Riemann Hypothesis for function fields, which states, among other things, that $z\mapsto \zeta_\theta(n,z)$ only has simple zeros, and was proven by Diaz-Vargas and Polanco-Chi in \cite{diaz-vargas_riemann_2016}. Alternatively, we can see this result as a special case of Theorem \ref{th : order vanishing t1 ts}.
    \end{proof}

\subsection{Definitions and tools}
\paragraph{Logarithm convergence.} 
If $x=(x_1,\ldots,x_n)^T\in \Cn(K_\infty)$ such that $v_\infty(x_i) > n-i-\frac{qn}{q-1}$ for all $i=1,\ldots,n$, then $\log_{\Cn}(x)$ converges in $\Lie_{\Cn}(K_\infty)$ (\cite[Proposition 2.4.3]{anderson_tensor_1990}). We also have the same result with the twist by $z$.  If $x=(x_1,\ldots,x_n)^T\in \widetilde{\Cn}(\widetilde{K_\infty})$ verifies $v_\infty(x_i)> n-i-\frac{qn}{q-1}$ for all $i=1,\ldots,n$, then  $\log_{\widetilde{\Cn}}(x)$ converges in $\Lie_{\widetilde{\Cn}}(\widetilde{K_\infty})$.

\paragraph{Carlitz multilogarithm.} Define $D_0 = L_0=1$, and for $i\in \N^*$,
$$D_i:=(\theta^{q^i}-\theta)(\theta^{q^i}-\theta^q)\ldots (\theta^{q^i}-\theta^{q^{i-1}}),$$
$$L_i:=(-1)^i(\theta^{q^i}-\theta)(\theta^{q^{i-1}}-\theta)\ldots (\theta^{q}-\theta).$$
Define the Carlitz \textit{multilogarithm} as 
$$\mathcal{L}_n = \sum_{k\geq 0} \frac{1}{L_k^n}\tau^k.$$
Consider the operator $\Delta:K_\infty\{\{\tau\}\}\rightarrow K_\infty\{\{\tau\}\}$ given by
$$\Delta \left(\sum_{i\geq 0}a_i\tau^i\right) := \sum_{i\geq 1}(\theta^{q^i}-\theta)a_i\tau^i.$$
The last component of $\log_{\Cn}$ can be expressed in terms of the multilogarithm. If $x=(x_1,\ldots,x_n)^T\in \Cn(K_\infty)$ is such that $\log_{\Cn}(x)$ converges, then we have that (\cite[Remark 7.6.2(1)]{thakur_function_2004})
$$\log_{\Cn}\begin{pmatrix}
    x_1 \\ x_2 \\ \vdots \\x_n 
\end{pmatrix} = \begin{pmatrix}
    * \\ \vdots \\ * \\ \sum_{i=0}^{n-1}(-\Delta)^i\mathcal{L}_n(x_{n-i})
\end{pmatrix}.$$

\paragraph{Hyperderivatives with respect to $\theta$.}  Define the $j$-th hyperderivative with respect to $\theta$ as the $\F_q$-linear map $\D_\theta^{(j)}:K \rightarrow K$ given by
$$\D_\theta^{(j)}(\theta^m)=\binom{m}{j}\theta^{m-j},\quad m\geq 0.$$
These operators can be used to compute the matrix $\partial(a)$ for any $a\in A$. If $a\in A$, define
$$d[a]:=\begin{pmatrix}
    a & \D_\theta^{(1)}(a) & \cdots & \D_\theta^{(n-1)}(a) \\
    & a & \ddots & \vdots \\
    && a & \D_\theta^{(1)}(a) \\
    &&& a
\end{pmatrix}.$$
Note that $d[\theta] = \partial(\theta)$, and both $d$ and $\partial$ are $\F_q$-algebra morphisms. It follows that 
$$d[a] = \partial(a)\qquad \text{ for all }a\in A.$$

\paragraph{Bernoulli-Carlitz numbers.} If $q-1\mid n$, then $\zeta(n)/\widetilde{\pi}^n \in K$, and the $n$-th Bernoulli-Carlitz number is
$$BC_n:=\frac{\zeta(n)\Gamma_n}{\widetilde{\pi}^n} \in K,$$
where $\Gamma_n$ denotes the Carlitz factorial of $n$ (see \cite[Section 9.1]{goss_basic_1996}).

\paragraph{Anderson-Thakur's special point.} We follow  \cite{anderson_tensor_1990}. Let $t$ be a variable. Define $\gamma_j(t)\in A[t]$ by
$$\gamma_0(t)=1,\qquad \gamma_j(y)=\prod_{l=1}^j(\theta^{q^j}-t^{q^l}) \text{ for } j\geq 1.$$
Define the Anderson-Thakur polynomials $\alpha_n(t)$ by the generating series
$$\sum_{n\geq 1}\frac{\alpha_n(t)}{\Gamma_n}x^n=x\left(1-\sum_{j\geq 0}\frac{\gamma_j(t)}{D_j}x^{q^j}\right)^{-1}.$$
We have that $\alpha_n(t)\in A[t]$ due to \cite[Equation 3.5.1]{anderson_tensor_1990}. Moreover, one can check that $\alpha_n(t)\in A[t^q]$. Write
$$\alpha_n(t)=\sum_{j=0}^sh_{n,j}t^{qj},\quad h_{n,j}\in A.$$
Then 
$$\mathfrak{z}_n := \sum_{j=0}^s\partial(h_{n,j})\log_{\Cn}(\theta^{qj}e_n)$$
is Anderson and Thakur's special point. It verifies 
$$\iota(\mathfrak{z}_n) = \Gamma_n\zeta(n), \quad \text{ and } \quad \exp_{\Cn}(\mathfrak{z}_n) \in \Cn(A),$$
where we recall that $\iota$ is the projection onto the last component.
One can readily check that the point
$$\widetilde{\mathfrak{z}_n} := \sum_{j=0}^s\partial(h_{n,j})\log_{\widetilde{\Cn}}(\theta^{qj}e_n)$$
verifies
$$\iota(\widetilde{\mathfrak{z}_n}) = \Gamma_n\zeta(n,z), \quad \text{ and } \quad \exp_{\widetilde{\Cn}}(\widetilde{\mathfrak{z}_n}) \in \widetilde{\Cn}(\widetilde{A}).$$

\begin{theorem} [Yu]\label{th : Yu}
\begin{enumerate}
    \item Let $x,y\in \Lie_{\Cn}(\overline{K_\infty})$ be two nonzero vectors. Suppose $\exp_{\Cn}(x),\exp_{\Cn}(y)\in \Cn(\overline{K})$ and $\iota(x) = \iota(y)$. Then $x=y$.
    \item Let $x,y\in U(\widetilde{\Cn}/A[z])$ be two nonzero vectors such that $\iota(x) = \iota(y)$. Then $x=y$.
\end{enumerate}
\end{theorem}
\begin{proof}
    For the first part, see \cite[Corollary 2.4]{yu_transcendence_1991}. For the second one, we refer to \cite[Lemma 5.4]{angles_special_2020}. 
\end{proof}

By Theorem \ref{th : class formula mod K tilde}, there exists some nonzero $a\in A[z]$ such that
    $$\partial(\zeta(n,z)a)\widetilde{b}\in U(\widetilde{\Cn}/A[z]),$$
    where $\widetilde{b}\in \LieEt(A[z])$ is taken to be an $\widetilde{A}$-generator of $\LieEt(\widetilde{A}) \cap \widetilde{W}$.
    By Yu's Theorem \ref{th : Yu},
    \begin{equation} \label{eq : zn <-> b}
        \widetilde{\mathfrak{z}_n} = \partial\left(\frac{\Gamma_n\zeta(n,z)}{\iota(\widetilde{b})}\right)\widetilde{b}.
    \end{equation}
It follows that $\widetilde{\mathfrak{z}_n} \in \widetilde{W}$, and $\mathfrak{z}_n \in W$.

\subsection{A formula for $\zeta_P^{(1)}(n)$}
Let 
$$n=l(q-1)$$
for some $l\in \N^+$. As discussed in the previous section, this means that $\zeta_P(n)=0$. Using a special point recently found by Pellarin \cite{pellarin_carlitz_2025} in combination with the work of the previous section, we will derive a formula for $\zeta_P^{(1)}(n)$. Set
    $$S_l(Y):= \sum_{i=0}^{l-1}(\theta-Y)^{l-1-i}(\theta^q-Y)^i = \frac{(\theta-Y)^l - (\theta^q-Y)^l}{\theta - \theta^q},$$
    and write
    $$S_l(Y) = \sum_{i=0}^{l-1}c_{l,i}Y^i.$$
Note that $\deg_\theta(c_{l,i})\leq q(l-i-1)$, and $\deg_\theta(c_{l,0})= q(l-1)$.
Set
$$J_l := \sum_{i=0}^{l-1} \Cn_{c_{l,i}}\begin{pmatrix}
    0\\
    \vdots\\
    0\\
    \theta^{qi}
\end{pmatrix} \quad \text{ and } \quad \widetilde{J_l} := \sum_{i=0}^{l-1} \widetilde{\Cn_{c_{l,i}}}\begin{pmatrix}
    0\\
    \vdots\\
    0\\
    \theta^{qi}
\end{pmatrix}.$$
Define also
$$j_l := \sum_{i=0}^{l-1} \partial(c_{l,i})\log_{\Cn}\begin{pmatrix}
    0\\
    \vdots\\
    0\\
    \theta^{qi}
\end{pmatrix}\quad \text{ and } \quad \widetilde{j_l} :=  \sum_{i=0}^{l-1} \partial(c_{l,i})\log_{\widetilde{\Cn}}\begin{pmatrix}
    0\\
    \vdots\\
    0\\
    \theta^{qi}
\end{pmatrix}.$$

\begin{theorem}[{\cite[Theorem C and Corollary D]{pellarin_carlitz_2025}}]\label{th : Pellarin point}
    Let $n=l(q-1)$, $l\geq 0$. 
    Then
    \begin{enumerate}
        \item $J_l$ is a nonzero $(\theta-\theta^q)$-torsion point for $\Cn$.
        \item For all $0\leq i\leq l-1$, we have
        $$BC_n = \frac{h_{n,i}}{c_{l,i}(\theta - \theta^q)},$$
        where we recall that the $h_{n,i}$ are the coefficients of the Anderson-Thakur polynomial.
    \end{enumerate}
\end{theorem}

\begin{prop}\label{prop : properties jl}
    Write $n=l(q-1)$. We have:
    \begin{enumerate}
    \item $\exp_{\Cn}(j_l)=J_l$ and $\exp_{\widetilde{\Cn}}(\widetilde{j_l})=\widetilde{J_l}.$
        \item $z_n = \partial(BC_n(\theta-\theta^q))j_l$ and $\widetilde{z_n} = \partial(BC_n(\theta-\theta^q))\widetilde{j_l}$.
        \item $j_l\in U(\Cn/A) \cap W$ and $\widetilde{j_l}\in U(\widetilde{\Cn}/\widetilde{A})\cap\widetilde{W}$.
        \item $\iota(j_l)=\frac{\widetilde{\pi}^n}{\theta-\theta^q}$ and $\iota(\widetilde{j_l)}=\frac{\Gamma_n\zeta(n,z)}{BC_n(\theta-\theta^q)}$.
    \end{enumerate}
\end{prop}
\begin{proof}
    The first point, follows directly from the functional equation of the exponential. 
    
    The second point is a direct consequence of the second point of Theorem \ref{th : Pellarin point}. Indeed,
    \begin{align*}
        z_n&=\sum_{i=0}^{l-1}\partial(h_{n,i})\log_{\Cn}(\theta^{qi}e_n) = \partial(BC_n(\theta-\theta^q))\sum_{i=0}^{l-1}\partial(c_{l,i})\log_{\Cn}(\theta^{qi}e_n) \\
        &= \partial(BC_n(\theta-\theta^q)) j_l.
    \end{align*}
    
    For the third point, it is clear that $\exp_{\Cn}(j_l)\in \Cn(A)$, and since $z_n\in W$, we have $j_l\in W$ by the previous point. The same arguments apply in the twisted case.
    
    Finally, using the second point, we have
    $$\iota(\widetilde{j_l})=\frac{1}{BC_n(\theta-\theta^q)}\cdot \iota(\widetilde{z_n}) = \frac{\Gamma_n\zeta(n,z)}{BC_n(\theta-\theta^q)}.$$
\end{proof}

\begin{remark} \label{rk : degree z}
Note that 
    $$\deg_\theta(c_{l,i})\leq q(l-1),$$
    with equality if and only if $i=0$.
    Therefore if $q\neq 2$, then
    $$\deg_\theta((\theta-\theta^q)c_{l,i}) \leq lq < 2l(q-1)=2n\qquad \text{ for all } 0\leq i\leq l-1.$$
    If $q=2$, then $lq = 2l(q-1)=2n$. We have
    $$\deg_\theta((\theta-\theta^q)c_{l,i}) < 2n\qquad \text{ for all } 1\leq i\leq l-1,$$
    and
    $$\deg_\theta((\theta-\theta^q)c_{l,0}) =2n.$$
\end{remark}

We will need the following combinatorial identities in our computations. 
\begin{lemma} \label{lem : combinatorial identities}
    The following identities involving binomial coefficients are true:
    \begin{enumerate}
        \item For any $a,b,c\in \N$, we have
        $$\binom{a}{b}\binom{a-b}{c}=\binom{a}{c}\binom{a-c}{b}.$$ \label{lem : combinatorial identities part 1}
        \item For any $a,b\in \N$ we have
        $$\sum_{i=0}^a(-1)^i\binom{a}{i}\binom{qi}{b}\equiv \begin{cases}
            (-1)^{a} \pmod{p} & b=qa\\
            0 \pmod{p} & b\neq qa
        \end{cases} .$$ \label{lem : combinatorial identities part 2}
    \end{enumerate}
\end{lemma} 

\begin{proof}
    \begin{enumerate}
        \item Assume $a\geq b+c$; otherwise both sides of the equality are zero. By a simple calculation, we see that both sides equal
        $$\frac{a!}{b!c!(a-b-c)!}.$$
        \item The sum we want to compute is the coefficient of $t^b$ in $(1-(1+t)^q)^a$. But $(1-(1+t)^q)^a \equiv (-t^q)^{a}\pmod{p}$, so we obtain the desired result. \qedhere
    \end{enumerate}
\end{proof}

Let 
$$\mathcal{X}_l:= \left(-\binom{l}{n-i}(\theta-\theta^q)^{l-(n-i)}\right)_{1\leq i \leq n} = -\begin{pmatrix}
    0 \\
    \vdots \\
    0 \\
    1 \\
    l (\theta-\theta^q)\\
    \vdots \\
    l(\theta-\theta^q)^{l-1}\\
    (\theta-\theta^q)^l
\end{pmatrix}\in \Mat_{n\times 1}(K).$$
We recall that $e_n=(0,0,\ldots,1)^T\in \Mat_{n\times 1}(K)$.

\begin{lemma} \label{lem : derivative J_l}
    Let $n=l(q-1)$. If $q\neq2$, then
    $$\left(\frac{d}{dz} \widetilde{\Cn_{\theta-\theta^q}}(\widetilde{J_l})\right)|_{z=1} = \mathcal{X}_l.$$
    If $q=2$,
    $$\left(\frac{d}{dz} \widetilde{\Cn_{\theta-\theta^q}}(\widetilde{J_l})\right)|_{z=1} = \mathcal{X}_l + e_n.$$
\end{lemma}

\begin{proof}
    \underline{Case $q\neq 2$.} By Remark \ref{rk : degree z}, we have that $\deg_\theta(\theta-\theta^q)c_{l,i}<2n$ for all $i$. Observe that if $a\in A$ is such that $\deg_\theta(a) < 2n$, then
    $$\widetilde{\Cn_a}(x e_n) = \partial(a)(xe_n) + z(\Cn_a(xe_n)-\partial(a)(xe_n))\quad \text{ for all } x\in K_\infty.$$
    We can therefore write
    \begin{align*}
        \widetilde{\Cn_{\theta-\theta^q}}(\widetilde{J_l}) &= \sum_{i=0}^{l-1}\left(\partial((\theta-\theta^q)c_{l,i})(\theta^{qi}e_n) + z \left(\Cn_{(\theta-\theta^q)c_{l,i}}(\theta^{qi}e_n) - \partial((\theta-\theta^q)c_{l,i})(\theta^{qi}e_n)\right) \right) .
    \end{align*} 
Hence
\begin{align*}
\frac{d}{dz} \widetilde{\Cn_{\theta-\theta^q}}(\widetilde{J_l}) &=  \sum_{i=0}^{l-1}\Cn_{(\theta-\theta^q)c_{l,i}}(\theta^{qi}e_n) - \sum_{i=0}^{l-1}\partial((\theta-\theta^q)c_{l,i}) (\theta^{qi}e_n) \\
& = \Cn_{\theta-\theta^q}(J_l) - \sum_{i=0}^{l-1}\partial((\theta-\theta^q)c_{l,i}) (\theta^{qi}e_n) \\
&= -\sum_{i=0}^{l-1}\partial((\theta-\theta^q)c_{l,i}) (\theta^{qi}e_n).
\end{align*}

We compute this element using hyperderivatives. Note that the $(n-r)$-th entry is equal to $-\D_\theta^{(r)}((\theta-\theta^q)S_l(Y))|_{Y=\theta^q}$.
\begin{align*}
    \D_\theta^{(r)}((\theta-\theta^q&)S_l(Y)) = \D_\theta^{(r)}((\theta-Y)^l) - \D_\theta^{(r)}((\theta^q-Y)^l) \\
    &= \sum_{i=0}^l(-1)^{i}\binom{l}{i}\D_\theta^{(r)}(\theta^{l-i})Y^i -  \sum_{i=0}^l(-1)^{l-i}\binom{l}{i}\D_\theta^{(r)}(\theta^{qi})Y^{l-i} \\
    &= \sum_{i=0}^l(-1)^i\binom{l}{i}\binom{l-i}{r}\theta^{l-i-r}Y^i -  \sum_{i=0}^l(-1)^{l-i}\binom{l}{i}\binom{qi}{r}\theta^{qi-r}Y^{l-i}.
\end{align*}
Hence
\begin{align*}
    \D_\theta^{(r)}((\theta-\theta^q)S_l(Y))|_{Y=\theta^q} &= \sum_{i=0}^l(-1)^i\binom{l}{i}\binom{l-i}{r}\theta^{l-i-r}\theta^{qi} - \sum_{i=0}^l(-1)^{l-i}\binom{l}{i}\binom{qi}{r}\theta^{ql-r}.
\end{align*}
We compute each sum separately. Using Lemma \ref{lem : combinatorial identities}.\ref{lem : combinatorial identities part 1}, we get for the first sum
\begin{align*}
    \sum_{i=0}^l(-1)^i\binom{l}{i}\binom{l-i}{r}\theta^{l-i-r}\theta^{qi} &= \sum_{i=0}^l(-1)^i\binom{l}{r}\binom{l-r}{i}\theta^{l-r-i}\theta^{qi}\\
    &= \binom{l}{r}(\theta-\theta^q)^{l-r}.
\end{align*}
Next, we want to use Lemma \ref{lem : combinatorial identities}.\ref{lem : combinatorial identities part 2} for the second sum. We obtain
\begin{align*}
    \sum_{i=0}^l(-1)^{l-i}\binom{l}{i}\binom{qi}{r}\theta^{ql-r} = \theta^{ql-r} \sum_{i=0}^l(-1)^{l-i}\binom{l}{i}\binom{qi}{r} = 0,
\end{align*}
since $r < l(q-1) \leq ql$. Thus, we have proven that 
$$\D_\theta^{(r)}((\theta-\theta^q)S_l(Y))|_{Y=\theta^q} = \binom{l}{r}(\theta-\theta^q)^{l-r}.$$
Therefore
$$\left(\frac{d}{dz} \widetilde{\Cn_{\theta-\theta^q}}(\widetilde{J_l})\right)|_{z=1} = \left(-\binom{l}{n-i}(\theta-\theta^q)^{l-(n-i)}\right)_{1\leq i \leq n} = \mathcal{X}_l.$$

\underline{Case $q=2$.} The proof is very similar, but we will get a small extra term when we compute the derivative. By Remark \ref{rk : degree z}, $\deg_\theta(\theta+\theta^q)c_{l,0}=2n$, while $\deg_\theta(\theta+\theta^q)c_{l,i}<2n$ for all $i\neq 0$. We note that, for a general $q$, if we take $a\in A$ with $\deg_\theta(a)=2n$, then for all $x\in K_\infty$ we have
$$\widetilde{\Cn_a}(x e_n) = \partial(a)(xe_n) + z(\Cn_a(xe_n)-\partial(a)(xe_n)-\lc(a)x^{q^2}e_n)+ z^2\lc(a)x^{q^2}e_n,$$
where $\lc(a)$ is the leading coefficient of $a$.
Returning to the case $q=2$, we can thus write
    \begin{align*}
        \widetilde{\Cn_{\theta+\theta^q}}(\widetilde{J_l}) &= \partial((\theta+\theta^q)c_{l,0})(e_n)
 + z \left(\Cn_{(\theta+\theta^q)c_{l,0}}(e_n) + \partial((\theta+\theta^q)c_{l,0})(e_n) + e_n\right) + z^2e_n \\ 
 &+\sum_{i=1}^{l+1}\left(\partial((\theta+\theta^q)c_{l,i})(\theta^{qi}e_n)
 + z \left(\Cn_{(\theta+\theta^q)c_{l,i}}(\theta^{qi}e_n) + \partial((\theta+\theta^q)c_{l,i})(\theta^{qi}e_n)\right)\right) .
    \end{align*} 
Hence
\begin{align*}
\frac{d}{dz} \widetilde{\Cn_{\theta+\theta^q}}(\widetilde{J_l}) &=  
\left(\sum_{i=0}^{l+1}\partial((\theta+\theta^q)c_{l,i}) (\theta^{qi}e_n)\right) + e_n + 2ze_n,
\end{align*}
and 
\begin{align*}
    \left(\frac{d}{dz} \widetilde{\Cn_{\theta-\theta^q}}(\widetilde{J_l})\right)|_{z=1} = \left(\sum_{i=0}^{l-1}\partial((\theta+\theta^q)c_{l,i}) (\theta^{qi}e_n)\right) + e_n.
\end{align*}
We conclude in the same way as in the case $q\neq2$.
\end{proof}

\begin{theorem}\label{th : formula derivative Cn}
    Suppose $n=l(q-1)$. If $q\neq 2$, then
    $$\zeta_P^{(1)}(n) = \frac{(P^n-1)BC_n}{P^n\Gamma_n}\iota(\Log_{\Cn,P}(\mathcal{X}_l)).$$
    If $q=2$,
    $$\zeta_P^{(1)}(n) = \frac{(P^n+1)BC_n}{P^n\Gamma_n}\iota(\Log_{\Cn,P}(\mathcal{X}_l+e_n)).$$
\end{theorem}
\begin{proof}
    We give the proof for $q\neq 2$; the case $q=2$ is completely analogous. We will follow the notation of the proof of Proposition \ref{prop : derivative P-adic zero iff torsion}, where we take $\widetilde{w} = \widetilde{j_l}$ and $a=\theta-\theta^q$. Taking the derivative of Equation \eqref{eq : LP=ga} and evaluating at $z=1$ gives us
    \begin{equation}\label{eq : explicit computation derivative}
        \zeta_P^{(1)}(n) = \left(\frac{d}{dz}\widetilde{g}\right)|_{z=1}\cdot\alpha,
    \end{equation}
    where we have used that $\widetilde{g}|_{z=1}=0$ whenever $q-1 \mid n$. \vspace{2mm}
    
    $\bullet$ \emph{Computation of $ \left(\frac{d}{dz}\widetilde{g}\right)|_{z=1}.$}
    Taking the derivative of Equation \eqref{eq : aw = gb} and evaluating at $z=1$ gives us
    $$\Log_{\Cn,P}\left(\left(\frac{d}{dz}\widetilde{\Cn_a}(\exp_{\widetilde{\Cn}}(\widetilde{w})\right)|_{z=1}\right) = \left(\frac{d}{dz}\partial(\widetilde{g})\right)|_{z=1} \cdot b.$$
    Note that by Lemma \ref{lem : derivative J_l}, the quantity inside $\Log_{\Cn,P}$ is just $\mathcal{X}_l$. Now looking at the last coordinate, we obtain
    $$\iota(\Log_{\Cn,P}(\mathcal{X}_l))= \left(\frac{d}{dz}\widetilde{g}\right)|_{z=1} \cdot \iota(b),$$
    and since $\iota(b)\neq 0$ (because $\iota:W\rightarrow K_\infty$ is injective), we get
    $$\left(\frac{d}{dz}\widetilde{g}\right)|_{z=1} = \frac{\iota(\Log_{\Cn,P}(\mathcal{X}_l))}{\iota(b)}.$$ \vspace{1mm}
    
    $\bullet$ \emph{Computation of $\alpha$.} 
    By Equation \eqref{eq : zn <-> b} and Proposition \ref{prop : properties jl} part 2,
    $$\partial(\theta-\theta^q)\widetilde{j_l} = \partial\left(\frac{\Gamma_n\zeta(n,z)}{BC_n\iota(\widetilde{b})}\right)\widetilde{b}.$$
    Formally in $K((z))$, Equation \eqref{eq : aw = gb} reads as $\partial(\theta-\theta^q)\widetilde{j_l} = \partial(\widetilde{g})\widetilde{b}$. Therefore
    $$\widetilde{g} = \frac{\Gamma_n\zeta(n,z)}{BC_n\iota(\widetilde{b})},$$
    and by Equation \eqref{eq : LP=ga},
    $$\widetilde{\alpha} = \frac{\zeta_P(n,z)}{\widetilde{g}} = \frac{BC_n\iota(\widetilde{b})}{\Gamma_n Z_P(\widetilde{\Cn}/\widetilde{A})},$$
    $$\alpha = \frac{BC_n\iota(b)}{\Gamma_n Z_P(\Cn/A)}.$$
    Substituting the computed values of $\alpha$ and the derivative of $\widetilde{g}$ back into Equation \eqref{eq : explicit computation derivative} gives us the desired result.
\end{proof}

\begin{remark}
    Write 
    $$\Cn_{P^n-1}(\mathcal{X}_l) = \begin{pmatrix}
        y_1 \\
        \vdots \\
        y_n
    \end{pmatrix}.$$
    Since 
    $$\Log_{\Cn,P}(\mathcal{X}_l) =  \partial(P^n-1)^{-1} \log_{\Cn,P}(\Cn_{P^n-1}(\mathcal{X}_l)),$$
    we have
    $$\iota (\Log_{\Cn,P}(\mathcal{X}_l)) = (P^n-1)^{-1} \sum_{i=1}^{n} (-\Delta)^{n-i}\mathcal{L}_n(y_i).$$
    Thus, if $q\neq 2$, then we can write
    $$\zeta_P^{(1)}(n) = \frac{BC_n}{P^n \Gamma_n}\sum_{i=1}^n (-\Delta)^{n-i}\mathcal{L}_n(y_i),$$
    and we can do the same thing for $q=2$.
\end{remark}

%%%%%%%%%%%%%%%%%%%%%%%%%%%%%%%%%%%%%%%%%%%%%%%%%%%%%%%%%%%%%%%%%%%%%%%%%%%%%%%%%%%%%%%%%%%%%%%%%%%%%%%%%%%%%%%%%%
%%%%%%%%%%%%%%%%%%%%%%%%%%%%%%%%%%%%%%%%%%%%%%%%%%%%%%%%%%%%%%%%%%%%%%%%%%%%%%%%%%%%%%%%%%%%%%%%%%%%%%%%%%%%%%%%%%
%%%%%%%%%%%%%%%%%%%%%%%%%%%%%%%%%%%%%%%%%%%%%%%%%%%%%%%%%%%%%%%%%%%%%%%%%%%%%%%%%%%%%%%%%%%%%%%%%%%%%%%%%%%%%%%%%%

\section{$P$-adic Pellarin $L$-series} \label{sec : Pellarin L-series}
We introduce some new notation. Let $t_1,\ldots,t_s$ be new variables. Then
\begin{itemize}
    \item $\underline{t}$ denotes the set of variables $t_1,\ldots,t_s$;
    \item with this notation, $\F_q(\underline{t})=\F_q(t_1,\ldots,t_s)$;
    \item $K_s:=K(\underline{t}),\quad A_s:=\F_q(\underline{t})[\theta],\quad \widetilde{K_s}:=K_s(z),\quad \widetilde{A_s}:=\F_q(z)A_s$;
    \item $K_{s,\infty}:=\Fqs ((\frac{1}{\theta})), \quad K_{s,P}:=\F_{q^{\deg(P)}}(\underline{t})((P));$
    \item $\widetilde{K_{s,\infty}}:=\Fq(\underline{t},z)((\frac{1}{\theta})), \quad \widetilde{K_{s,P}}:=\F_{q^{\deg(P)}}(\underline{t},z)((P));$
    \item $F_s:=F(\underline{t}),\quad  \OFs:=\Fqs\OF, \quad\widetilde{F_s}:=F(\underline{t})(z),\quad  \widetilde{\OFs}:=\Fq(\underline{t},z)\OF$;
    \item $F_{s,\infty}:=F_s\otimes_{K_s} K_{s,\infty},\quad F_{s,P}:=F_s\otimes_{K_s} K_{s,P};$
    \item $\widetilde{F_{s,\infty}}:=\widetilde{F_s}\otimes_{\widetilde{K_s}} \widetilde{K_{s,\infty}},\quad \widetilde{F_{s,P}}:=\widetilde{F_s}\otimes_{\widetilde{K_s}} \widetilde{K_{s,P}};$
    \item $\T(K_{s,v}):= \left\{\sum_{i=0}^\infty a_iz^i \in K_{s,v}[[z]] : v(a_i)\rightarrow \infty\right\},\quad v=v_\infty,v_P\,;$
    \item $ \T(F_{s,v}):= \left\{\sum_{i=0}^\infty a_iz^i \in F_{s,v}[[z]] : v(a_i)\rightarrow \infty\right\},\quad v=v_\infty,v_P\,;$
    \item for $v=v_\infty,v_P$, we extend the valuation $v$ from $\widetilde{K_v}$ to $\widetilde{K_{s,v}}$ by means of the Gauss valuation on the variables $t_1,\ldots,t_s$.
\end{itemize}

 For each $i=1,\ldots,s$, let $\chi_{t_i}:A\rightarrow \Fq[t_1,\ldots,t_s]$ be the $\F_q$-algebra morphism defined by $\chi_{t_i}(\theta)=t_i.$ Consider the following versions of Pellarin's $L$-series:
$$L(\chi_{t_1},\ldots,\chi_{t_s},n,z):=\sum_{d\geq 0}\sum_{a\in A_{+,d} }\frac{\chi_{t_1}(a)\ldots\chi_{t_s}(a)}{a^n}z^d \in \T(\Ksi),$$
$$L_P(\chi_{t_1},\ldots,\chi_{t_s},n,z):=\sum_{d\geq 0}\sum_{\substack{a\in A_{+,d} \\ P\nmid a}}\frac{\chi_{t_1}(a)\ldots\chi_{t_s}(a)}{a^n} z^d \in \T(K_{s,P}),$$
$$L(\chi_{t_1},\ldots,\chi_{t_s},n):=L(\chi_{t_1},\ldots,\chi_{t_s},n,1),$$
$$L_P(\chi_{t_1},\ldots,\chi_{t_s},n):=L_P(\chi_{t_1},\ldots,\chi_{t_s},n,1).$$
We refer to \cite[Section 6]{lucas_P-adic_2026} for more details about the convergence of the sums. Demeslay showed in \cite[Section 4.2]{demeslay_class_2022} that these $L$-series are in fact the $L$-series of a certain multivariable deformation of the tensor power of the Carlitz module.

An Anderson $A_s$-module $E/\OFs$ of dimension $n$ is a morphism of $\F_q(\underline{t})$-algebras
\begin{align*}
    E  : A_s&\rightarrow \Mat_{n\times n}(\OFs)\{\tau\} \\
    \theta & \mapsto E_\theta = \partial_E(\theta) + \sum_{i=1}^{D} E_{i}\tau^i
\end{align*}
such that $(\partial_E(\theta)-\theta\I_n)^n=0$. Its $z$-deformation is the morphism of $\F_q(\underline{t},z)$-algebras
\begin{align*}
    \widetilde{E}  : \widetilde{A_s} &\rightarrow \Mat_{n\times n}(\widetilde{\OFs})\{\tau\} \\
\theta  &\mapsto \widetilde{E}_\theta = \partial_E(\theta) + \sum_{i=1}^{D} E_{i}\tau^iz^i.
\end{align*}
The exponential and logarithm are defined in the usual way. Set
$$U(E/\OFs):=\{x\in \LieE(F_{s,\infty}) : \exp_E(x)\in E(\OFs)\},$$
$$U(\widetilde{E}/\widetilde{\OFs}):=\{x\in \Lie_E(\widetilde{F_{s,\infty}}):\expEt(x)\in \widetilde{E}(\widetilde{\OFs})\}.$$

The results of Section \ref{sec : ord van L_P} still hold in this setting. The proofs are completely analogous. As before, denote by $L$ the perfection of $F$, and let $L_s=L(\underline{t})$. We say that a module $E/\OFs$ is $A_s$-finite of rank $r$ if the corresponding dual motif $N_E(L_s)$ is a free $L_s[t]$-module of rank $r$.

\begin{theorem}\label{th : class formula mod K tilde with s}
    Let $E/\OF$ be an $A_s$-finite Anderson module of rank $r$ such that $\sigma  N_E(L_s)\subseteq (t-\theta)N_E(L_s)$. Then there exists a sub-$\widetilde{K_{s,\infty}}$-vector space $\widetilde{W}_s$ of $\LieEt(\widetilde{\Fsi})$ of dimension $rm$ such that 
    \begin{enumerate}
        \item $U(\widetilde{E}/\widetilde{O_{F,s}})\cap \widetilde{W}_s$ and $\LieEt(\widetilde{O_{F,s}})\cap \widetilde{W}_s$ are $\widetilde{A_s}$-lattices in $\widetilde{W}_s.$
        \item We have $$L(\widetilde{E}/\widetilde{O_{F,s}}) = [\LieEt(\widetilde{O_{F,s}})\cap \widetilde{W}_s : U(\widetilde{E}/\widetilde{O_{F,s}})\cap \widetilde{W}_s]_{\widetilde{A_s}}\cdot \widetilde{\alpha}$$
        for some $\alpha \in \widetilde{K_s}^*$.
    \end{enumerate}
    Moreover, $W_s:=\left(\widetilde{W}_s \,\cap\, \LieEt(\T(F_{s,\infty}))\right)|_{z=1}$ is a sub-$K_{s,\infty}$-vector space  of $\LieE(\Fsi)$ of dimension $rm$ such that 
    \begin{enumerate}
        \item $U(E/O_{F,s})\cap W_s$ and $\LieE(O_{F,s})\cap W_s$ are $A_s$-lattices in $W_s.$
        \item We have $$L(E/O_{F,s}) = [\LieE(O_{F,s})\cap W_s : U(E/O_{F,s})\cap W_s]_{A_s}\cdot \alpha$$
        for some $\alpha \in K_s^*$.
    \end{enumerate}
\end{theorem}
\begin{proof}
    Same as Theorem \ref{th : class formula mod K tilde}.
\end{proof}

\begin{theorem}\label{th : P-adic formula W with ts}
    Let $E/\OF$ be an $A_s$-finite Anderson module such that $\sigma  N_E(L_s)\subseteq (t-\theta)N_E(L_s)$. We have the following equality over $\widetilde{K_{s,P}}$:
    $$L_P(\widetilde{E}/\widetilde{\OFs})= [\LieEt(\widetilde{\OFs})\cap \widetilde{W}_s : U(\widetilde{E}/\widetilde{\OFs})\cap \widetilde{W}_s]_{\widetilde{A_s},P} \cdot \widetilde{\alpha}$$
    for some $\alpha\in \widetilde{K_s}^*$. Over $K_{s,P}$, we have
    $$L_P(E/\OFs)= [\LieE(\OFs)\cap W_s : U(E/\OFs)\cap W_s]_{A_s,P}\cdot \alpha$$
    for some $\alpha\in K_s^*$.
\end{theorem}
\begin{proof}
    Same as Theorem \ref{th : P-adic formula W}.
\end{proof}

\begin{theorem}\label{th : vanishing LP ts}
    Let $E/A_s$ be an $A_s$-finite Anderson module of rank $1$ such that $\sigma  N_E(L_s)\subseteq (t-\theta)N_E(L_s)$. The following statements are true. 
    \begin{enumerate}
        \item $L_P(E/A_s)=0$
    if and only if
    $\exp_E:\LieE(K_{s,\infty})\rightarrow E(K_{s,\infty})$
    is not injective.
        \item The order of vanishing of $L_P(\widetilde{E}/\widetilde{A_s})$ at $z=1$ does not depend on $P$.
    \end{enumerate}
\end{theorem}
\begin{proof}
    Identic to the proofs of Corollary \ref{cor : vanishing P-adic L} and Theorem \ref{th : vanishing does not depend P}.
\end{proof}

For any nonzero $\alpha\in A_s$, consider the Anderson module 
$$E_\alpha:A_s\rightarrow \Mat_{n\times n}(A_s),\qquad (E_\alpha)_\theta=\partial(\theta)+N_\alpha\tau,$$
where
$$\partial(\theta)=\begin{pmatrix}
    \theta &1 & 0 & \ldots & 0 \\
    0 & \theta & 1 & \ldots & 0 \\
    \vdots&&\vdots&&\vdots \\
    0 & \ldots & 0 &\theta &1 \\
    0 &\ldots &\ldots &0& \theta
\end{pmatrix}, \qquad N_\alpha = \begin{pmatrix}
    0 &\ldots&\ldots &0\\
    \vdots&&&\vdots\\
    0 &\ldots&\ldots &0\\
    \alpha & 0 & \ldots & 0
\end{pmatrix}.$$

We extend our module $\Cn$ $\F_q(\underline{t})$-linearly, so that  $E_1=\Cn$. In the same way as for $\Cn$, one can check that $N_{E_\alpha}(L_s)=L_s[t]e_n^T$, and so $E_\alpha$ is $A_s$-finite of rank $1$. We also clearly have $\sigma N_E(L_s)\subseteq (t-\theta)N_E(L_s)$. Thus, the module $E_{\alpha}$ verifies the conditions of the preceding theorem.

In the following, we will consider 
$$\alpha:=(t_1-\theta)\ldots(t_s-\theta).$$
As shown in \cite[Section 4.2]{demeslay_class_2022}, the $L$-series attached to $E_\alpha$ are exactly the Pellarin $L$-series: 
    $$L(E_\alpha/A_s)=\prod_{P}\frac{[\Lie_{E_\alpha}(A_s/PA_s)]_{A_s}}{[E_\alpha(A_s/PA_s)]_{A_s}} = L(\chi_{t_1},\ldots,\chi_{t_s},n),$$
    $$L_P(E_\alpha/A_s)=\prod_{Q \neq P}\frac{[\Lie_{E_\alpha}(A_s/QA_s)]_{A_s}}{[E_\alpha(A_s/QA_s)]_{A_s}} = L_P(\chi_{t_1},\ldots,\chi_{t_s},n),$$
and the same is true with a twist by $z$.                 

Let
$$\omega_\alpha := \sqrt[q-1]{(-\theta)^s}\prod_{i\geq 0} \tau^i\left(\frac{(-\theta)^s}{(t_1-\theta)\ldots(t_s-\theta)}\right).$$
Note that $v_\infty(\tau^i(1-(-\theta)^s/\alpha))\rightarrow\infty$, so the infinite product converges. One can check that
$$\tau(\omega_\alpha)=\alpha\omega_\alpha.$$
\begin{prop} \label{prop : properties E_a}
We have the following properties about $E_\alpha$.
    \begin{enumerate}
        \item $\Cn_a \omega_\alpha = \omega_\alpha (E_\alpha)_a$ for all $a\in A_s$.
        \item $\exp_{E_\alpha} = \omega_\alpha^{-1} \exp_{\Cn}\omega_\alpha$.
        \item $\ker \exp_{E_\alpha}$ is a free $A_s$-module of rank $1$, generated by $\pi_n / \omega_\alpha$, whose last component is $\widetilde{\pi}^n/\omega_\alpha$.
        \item Over $K_{s,\infty}$, the exponential 
        $$\exp_{E_\alpha}:\Lie_{E_\alpha}(K_{s,\infty})\rightarrow E_\alpha(K_{s,\infty})$$ 
        is not injective if and only if 
        $$n\equiv s \pmod{q-1}.$$
    \end{enumerate}
\end{prop}
\begin{proof}
    The first three points are direct computations (\cite[Proposition 4.7]{demeslay_class_2022}). For the last point, note that $\pi_n/\omega_\alpha\in \Lie_{E_\alpha}(K_{s,\infty})$ if and only if $\sqrt[q-1]{(-\theta)^{n-s}}\in K_{\infty}$.
\end{proof}
\begin{cor} \label{cor : vanishing LP t1 ts}
    $L_P(E_\alpha/A_s)=0$ if and only if $n\equiv s \pmod{q-1}$.
\end{cor}
\begin{proof}
    It is a consequence of the first part of Theorem \ref{th : vanishing LP ts} and the fourth point of Proposition \ref{prop : properties E_a}.
\end{proof}

Denote by $\chi_{t_i}':A\rightarrow \F_q(\underline{t})$ the $\F_q$-algebra morphism given by $\chi_{t_i}'(\theta)=1/t_i$. Let $\chi_s:=\chi_{t_1}\ldots\chi_{t_s}$ and $\chi_s':=\chi_{t_1}'\ldots\chi_{t_s}'$.

\begin{lemma} \label{lem : L chi entire}
    For any $(\eta_1,\ldots,\eta_s)\in \overline{\F_q}^s$ and any $x\in K_{s,\infty}$, the series
    $$L(\chi_{t_1},\ldots,\chi_{t_s},n,z)|_{t_i=\eta_i,z=x}$$
    converges. The same is then true for $L(\chi'_{t_1},\ldots,\chi'_{t_s},n,z)$.
\end{lemma}
\begin{proof}
    See \cite[Proposition 6]{angles_functional_2014}.
\end{proof}

\begin{theorem} \label{th : order vanishing t1 ts}
    If $n\equiv s \pmod{q-1}$, then 
        $$\ord_{z=1}L_P(\chi_{t_1},\ldots,\chi_{t_s},n,z) = 1.$$
\end{theorem}
\begin{proof}
    By Theorem \ref{th : vanishing LP ts}, the order of vanishing does not depend on $P$, so it is enough to check that $L_P(\chi_{t_1},\ldots,\chi_{t_s},n,z)$ has a simple zero at $z=1$ for $P=\theta$. Consider the continuous $\Fqs$-linear map given by
    \begin{align*}
        \Psi : K_{s,\theta} &\longrightarrow K_{s,\infty}\\
        \theta &\longmapsto 1/\theta,
    \end{align*}
    which is an isomorphism of valued fields. We have
    \begin{multline*}
        \Psi(  L_\theta(\chi_{t_1},\ldots,\chi_{t_s},n,z)) = \Psi\left(\sum_{d\geq 0}\sum_{\substack{a\in A_{+,d}\\ \theta\nmid a}} \frac{\chi_s(a)}{a^n}z^d\right) \\
        = \sum_{d\geq 0}\sum_{\substack{a\in A_{+,d}\\ \theta\nmid a}} \frac{\chi_s(a)}{a(\frac{1}{\theta})^n}z^d 
        = \sum_{d\geq 0}\sum_{\substack{a\in A_{+,d} \\\theta\nmid a}} \frac{\theta^{dn}(\chi_s(a)\chi_s(\theta)^{-d})\chi_s(\theta)^{d}}{(\theta^d a(\frac{1}{\theta}))^n}z^d \\
        =\sum_{d\geq 0}\sum_{\substack{b\in A_d \\b(0)=1}} \frac{\chi'_s(b)}{ b^n}(\chi_s(\theta)\theta^nz)^d 
        = \sum_{d\geq 0}\sum_{\substack{a\in A_{+,d} \\ \theta \nmid a}} \frac{\chi'_s(a)}{ a^n}(\chi_s(\theta)\theta^nz)^d.
    \end{multline*}
The last equality is a consequence of the fact that if $b\in A$, $\lc(b)\in \F_q^*$ denotes the leading coefficient of $b$, and $n\equiv s \pmod{q-1}$, then
$$\frac{\chi_s(b)}{b^n} = \frac{\chi_s(\lc(b))\,\chi_s\left(\frac{b}{\lc(b)}\right)}{\lc(b)^n\left(\frac{b}{\lc(b)}\right)^n} = \frac{\lc(b)^{s}}{\lc(b)^n}\cdot\frac{\chi_s\left(\frac{b}{\lc(b)}\right)}{\left(\frac{b}{\lc(b)}\right)^n}
=\frac{\chi_s\left(\frac{b}{\lc(b)}\right)}{\left(\frac{b}{\lc(b)}\right)^n}.$$
Formally in $K_s[[y]]$, we have that
$$\sum_{d\geq 0}\sum_{\substack{a\in A_{+,d}\\ \theta \nmid a}} \frac{\chi'_s(a)}{ a^n}y^d = \left(1-y\frac{\chi'_s(\theta)}{\theta^n}\right)\sum_{d\geq 0}\sum_{\substack{a\in A_{+,d} }} \frac{\chi'_s(a)}{ a^n}y^d,$$
and therefore
\begin{align*}
    \Psi(L_P(\chi_{t_1},\ldots,\chi_{t_s},n,z)) &= (1-z)\sum_{d\geq 0}\sum_{\substack{a\in A_{+,d} }} \frac{\chi'_s(a)}{ a^n}(\chi_s(\theta)\theta^nz)^d.
\end{align*}
We note that the sum on the right hand-side converges at $z=1$ thanks to Lemma \ref{lem : L chi entire}. Hence, we just have to check that the sum does not vanish at $z=1$. 
We have
\begin{align*}
    \sum_{\substack{a\in A_{+,d} }} \frac{\chi'_s(a)}{ a^n}(\chi_s(\theta)\theta^n)^d &= \frac{\chi_s(\theta)^d\theta^{dn}\sum_{a\in A_{+,d}}\chi'_s(a)\prod_{b\in A_{+,d}\setminus\{a\}}b^n}{\prod_{a\in A_{+,d}}a^n} \\
    &=\frac{\chi_s(\theta)^d \theta^{q^ddn} \sum_{a\in A_{+,d}}\chi'_s(a) +  \text{ terms of lower degree}}{\theta^{q^ddn} + \text{ terms of lower degree}} \\
    &=\frac{\chi_s(\theta)^d  \sum_{a\in A_{+,d}}\chi'_s(a) +  \text{ terms of lower degree}}{1 + \text{ terms of lower degree}}.
\end{align*}
We conclude thanks to Lemma \ref{lem : character sums t1 ts}, which we state down below:
\begin{align*}
    \sum_{d\geq 0} \sum_{\substack{a\in A_{+,d} }}& \frac{\chi'_s(a)}{ a^n}(\chi_s(\theta)\theta^n)^d \\
    &= \underbrace{\left(\sum_{d=0}^{\lfloor\frac{s}{q-1}\rfloor} \chi_s(\theta)^d  \sum_{a\in A_{+,d}}\chi'_s(a)\right)}_\text{$\neq 0$} + \text{ terms of lower degree}\\& \neq 0.
\end{align*}
\end{proof}

\begin{lemma}\label{lem : character sums t1 ts}
    Write $\chi_s=\chi_{t_1},\ldots,\chi_{t_s}$. Let 
    $$R_d(s):=\chi_s(\theta)^d\sum_{a\in A_{+,d}}\chi'_s(a).$$
    Then 
    $$R_d(s)=0 \iff d(q-1)>s.$$
    Moreover, for any $D\in \N^*$ we have
    $$\sum_{d=0}^D R_d(s) \neq 0.$$
\end{lemma}
\begin{proof}
    We follow very closely the proof of \cite[Lemma 30]{angles_universal_2013}. Recall that 
    $$\sum_{\lambda\in \F_q}\lambda^l = \begin{cases}
        -1 & l\equiv 0 \pmod{q-1} \text{ and } l\geq 1 \\
        0 & \text{otherwise}
    \end{cases}.$$
    Since
    $$R_d(s) = \sum_{a_1,\ldots,a_d\in\F_q}\prod_{i=1}^s(1 + a_1t_i+ \ldots + a_dt_i^d),$$
    the coefficient of $t_1^{v_1}\ldots t_s^{v_s}$, $0\leq v_i \leq d$, is
    $$c_{v_1,\ldots,v_s}=\sum_{a_1,\ldots,a_d\in\F_q}a_{v_1}\ldots a_{v_s}$$
    if we define $a_0=1$. Set
    $$\mu_i:=\#\{j: v_j=i\}$$
    for $i=0,\ldots,d$. Note that $\sum_{i=0}^d\mu_i=s$, and therefore $\sum_{i=1}^d\mu_i\leq s$.
    We can now rewrite the coefficients as
    $$c_{v_1,\ldots,v_s} = \left(\sum_{a_1\in\F_q}a_1^{\mu_1}\right)\ldots\left(\sum_{a_d\in\F_q}a_d^{\mu_d}\right). $$
    By the remark made at the beginning of the proof, $c_{v_1,\ldots,v_s}$ does not vanish if and only if all $\mu_i$ are nonzero multiples of $q-1$. If $s<d(q-1)$, then for all $0\leq v_1,\ldots,v_s\leq d$, we have always at least one $\mu_i<q-1$, $i>0$, since $\sum_{i=1}^d\mu_i\leq s < d(q-1)$. Therefore $R_d(s)=0$. On the other hand, if $s\geq d(q-1)$, then any coefficient $c_{v_1,\ldots,v_s}$ such that $\mu_1=\ldots=\mu_d=q-1$ does not vanish. 
    
    To prove the second part of the lemma, it is enough to note that $R_0(s)=1$ and $R_d(s)|_{t_1=0,\ldots t_s=0}=0$ for $d\geq 1$.
\end{proof}

%%%%%%%%%%%%%%%%%%%%%%%%%%%%%%%%%%%%%%%%%%%%%%%%%%%%%%%%%%%%%%%%%%%%%%%%%%%%%%%%%%%%%%%%%%%%%%%%%%%%%%%%%%%%%%%%%%
%%%%%%%%%%%%%%%%%%%%%%%%%%%%%%%%%%%%%%%%%%%%%%%%%%%%%%%%%%%%%%%%%%%%%%%%%%%%%%%%%%%%%%%%%%%%%%%%%%%%%%%%%%%%%%%%%%
%%%%%%%%%%%%%%%%%%%%%%%%%%%%%%%%%%%%%%%%%%%%%%%%%%%%%%%%%%%%%%%%%%%%%%%%%%%%%%%%%%%%%%%%%%%%%%%%%%%%%%%%%%%%%%%%%%

\section{$P$-adic Dirichlet-Goss $L$-series} \label{sec : Dirichlet characters}

Let $\eta  \in \overline{\F_q}$ and let $P\in A_+$ be an irreducible polynomial such that $P(\eta)=0$. The morphism of $\F_q$-algebras $\chi_\eta:A \rightarrow\overline{\F_q}$, $a\mapsto a(\eta)$ induces an injective morphism of groups
$$\chi_\eta: \left(\frac{A}{PA}\right)^\times  \rightarrow\overline{\F_q}^*.$$
In general, a Dirichlet character is a morphism of groups
$$\chi:\left(\frac{A}{fA}\right)^\times \rightarrow\overline{\F_q} ^*$$
for some $f\in A_+$. If $(a,f)\neq 1$, we set $\chi(a)=0$. For every Dirichlet character $\chi$, there exist unique elements (up to permutation) $\eta_1,\ldots,\eta_s \in \overline{\F_q}$ such that
$$\chi(a)=\prod_{k=1}^s\chi_{\eta_k}(a)\quad \forall a\in A.$$
We say that $s$ is the \textit{type} of $\chi$. 
\begin{comment}
    Let $P_k\in A_+$ be the prime of $A$ such that $P_k(\zeta_k)=0$. We define the conductor of $\chi$ as $\cond(\chi)=\lcm(P_1,\ldots,P_s)$. We will assume that all characters are primitive, meaning that they are defined modulo their conductor. In other words, we assume $f=\cond(\chi)$.
\end{comment}

As usual, fix $P\in A_+$ a monic irreducible polynomial. If $\chi$ is a Dirichlet character, then the $P$-adic $L$-series associated to $\chi$ is 
$$L_P(n,\chi):=\sum_{d\geq 0}\sum_{\substack{a\in A_{+,d}\\P\nmid a}}\frac{\chi(a)}{a^n},$$
and the twisted $L$-series is
$$L_P(n,\chi,z):=\sum_{d\geq 0}\sum_{\substack{a\in A_{+,d}\\P\nmid a}}\frac{\chi(a)}{a^n}z^d.$$
If $\chi$ is the character of type $s$ associated to the values $\eta_1,\ldots,\eta_s \in \overline{\F_q}$, then
$$L_P(n,\chi,z) = L_P(\chi_{t_1},\ldots,\chi_{t_s},n,z)|_{t_1=\eta_1,\ldots , t_s=\eta_s}.$$

Let $\F_q(\chi):=\F_q(\eta_1,\ldots,\eta_s)$. Define
$$a(\chi):= (\eta_1-\theta)\ldots(\eta_s-\theta) \in A \otimes \F_q(\chi).$$
The map 
\begin{align*}
    E_{\alpha(\chi)}:A\otimes \Fqchi &\rightarrow\Mat_{n \times n}(K\otimes \F_q(\chi)) \\
    \theta &\mapsto \partial(\theta)+ N_{\alpha(\chi)}\tau,
\end{align*}
endows $(K\otimes \F_q(\chi))^n$ with an $A\otimes \Fqchi$-module structure, where $\partial(\theta)$ and $N_{\alpha(\chi)}$ are defined as in Section $\ref{sec : Pellarin L-series}$, and where the Frobenius only acts on the first component: $\tau(x\otimes y)=\tau(x)\otimes y$ for $x\in K,y\in \Fqchi$. Clearly
$$E_{\alpha(\chi)} = E_\alpha|_{t_1=\eta_1,\ldots,t_s=\eta_s}.$$
We set 
$$L_P(E_{\alpha(\chi)}/A\otimes\Fqchi):=L_P(E_\alpha/A_s)|_{t_1=\eta_1,\ldots,t_s=\eta_s},$$ so that
$$L_P(n,\chi)=L_P(E_{\alpha(\chi)}/A\otimes\Fqchi).$$
Let 
$$\exp_{E_{\alpha(\chi)}}:= (\exp_{E_\alpha})|_{t_1=\eta_1,\ldots,t_s=\eta_s}.$$

If $v$ is either the $\infty$-adic or a $P$-adic place, then we let $K_v\hat{\otimes}\Fqchi$ be the $v$-adic completion of $K_v\otimes\Fqchi$.
\begin{prop} \label{prop : expenential chi}
    The exponential
    $$\exp_{E_{\alpha(\chi)}}:\Lie_{E_{\alpha(\chi)}}(K_\infty\hat{\otimes}\Fqchi) \rightarrow E_{\alpha(\chi)}(K_\infty\hat{\otimes}\Fqchi)$$
    is not injective if and only if $s(\chi)\equiv n \pmod{q-1}$.
\end{prop}
\begin{proof}
    Same proof as Proposition \ref{prop : properties E_a}.
\end{proof}
Now Corollary \ref{cor : vanishing P-adic L} still holds in this context, which gives us the following result.
\begin{cor} \label{cor : vanishing L Dirichlet}
    Let $\chi$ be a Dirichlet character. Then 
    $$L_P(n,\chi)=0$$
    if and only if 
    $$n\equiv s(\chi) \pmod{q-1}.$$
\end{cor}

\begin{defin}
    We say that a property holds for \textit{almost all} characters of type $s$ if there exists a nonzero polynomial $G\in {\F_q}[X_1,\ldots,X_s]$ such that the property holds for all characters corresponding to $(\eta_1,\ldots,\eta_s)\in \overline{\F_q}^s$ with $G(\eta_1,\ldots,\eta_s)\neq 0$.
\end{defin}
\begin{theorem} \label{th : ord vanishing dirichlet characters}
    Take $s\in \N$ such that $n\equiv s \pmod{q-1}$. For almost all Dirichlet characters $\chi$ of type $s$, we have
    $$\ord_{z=1}L_P(n,\chi,z) = 1.$$
    If $s<q-1$, then this is true for all characters of type $s$.
\end{theorem}
\begin{proof}
    Assume $\chi(\theta)\neq 0$. By Theorem \ref{th : vanishing does not depend P}, the order of vanishing does not depend on $P$, so it is enough to check that $L_P(n,\chi,z)$ has a simple zero at $z=1$ for $P=\theta$. We use again the isomorphism of valued fields
    \begin{align*}
        \Psi : K_\theta \hat{\otimes} \Fqchi &\longrightarrow K_\infty\hat{\otimes}\Fqchi\\
        \theta \otimes 1 &\longmapsto 1/\theta \otimes 1.
    \end{align*}
    Replicating the proof of Theorem \ref{th : order vanishing t1 ts}, we get
\begin{align*}
    \Psi(L_\theta(n,\chi,z)) &= (1-z)\sum_{d\geq 0}\sum_{\substack{a\in A_{+,d} }} \frac{\chi'(a)}{ a^n}(\chi(\theta)\theta^nz)^d,
\end{align*}
where $\chi'(a) = a(\eta_1^{-1})\ldots a(\eta_s^{-1})$. Again, the sum on the right-hand side converges at $z=1$ due to Lemma \ref{lem : L chi entire}. Hence, we just have to check that the sum does not vanish at $z=1$. Let
$$G(t_1,\ldots,t_s):=\sum_{d=0}^{\lfloor\frac{s}{q-1}\rfloor} \chi_{s}(\theta)^d  \sum_{a\in A_{+,d}}\chi_{s}'(a)\in \Fq[t_1,\ldots,t_s],$$
which is nonzero by Lemma \ref{lem : character sums t1 ts}. Then
\begin{align*}
    \sum_{d\geq 0} \sum_{\substack{a\in A_{+,d} }}& \frac{\chi'(a)}{ a^n}(\chi(\theta)\theta^n)^d =G(\eta_1,\ldots,\eta_s) + \text{ terms of lower degree}
\end{align*}
does not vanish whenever $G(\eta_1,\ldots,\eta_s)\neq 0$. If $s<q-1$, then $G(\eta_1,\ldots,\eta_s)=1$ is always nonzero.
\end{proof}

\bibliographystyle{alpha}

\begin{small} 
  \bibliography{biblio}
\end{small}

\end{document}